\documentclass[11pt,reqno]{amsart}
\usepackage{verbatim}
\usepackage{marginnote}
\usepackage{amsmath}
\usepackage{amsfonts}
\usepackage{amssymb}
\usepackage{amsthm}
\usepackage{graphicx}
\usepackage{setspace}
\usepackage{gensymb}
\usepackage{enumerate}
\usepackage{mathrsfs}
\usepackage{empheq} 
\usepackage{graphics}
\usepackage{epsfig} 
\usepackage{bbm}
\usepackage[colorlinks=true, citecolor=blue]{hyperref}
\usepackage[left=2.5cm,top=3.25cm,right=2.5cm, bottom=3cm]{geometry}
\usepackage{graphicx, psfrag}
\usepackage{stackrel}

\usepackage{pgfplots}

\DeclareMathOperator{\spt}{spt}

\DeclareMathOperator{\esssup}{esssup}

\newcommand{\Lam}{\Lambda}
\newcommand{\w}{\omega}

\newcommand{\tht}{\theta}
\newcommand{\Tht}{\Theta}
\newcommand{\al}{\alpha}
\newcommand{\eps}{\epsilon}
\newcommand{\be}{\beta}
\newcommand{\ph}{\varphi}
\newcommand{\De}{\Delta}
\newcommand{\de}{\delta}
\newcommand{\s}{\sigma}

\newcommand{\gam}{\gamma}

\newcommand{\kap}{\kappa}
\newcommand{\lam}{\lambda}

\newcommand{\Ph}{\Phi}
\newcommand{\veps}{\varepsilon}
\newcommand{\z}{\zeta}

\newcommand{\ind}{\mathbbm{1}}

\newcommand{\N}{\mathbb{N}}
\newcommand{\R}{\mathbb{R}}

\newcommand{\T}{\mathbb{T}}
\newcommand{\Z}{\mathbb{Z}}

\newcommand{\A}{\mathcal{A}}

\newcommand{\del}{\nabla}
\newcommand{\bdy}{\partial}

\newcommand{\Ri}{\mathcal{R}}

\newcommand{\Sq}{\mathscr{S}}

\newcommand{\til}[1]{\widetilde{#1}}

\newcommand{\lp}{\triangle}

\newcommand{\lpj}{\triangle_j}
\newcommand{\lpjc}{\til{\triangle}_j}
\newcommand{\lpk}{\triangle_k}

\newcommand{\lpl}{\triangle_\ell}

\newcommand{\smod}{\setminus}

\newcommand{\sub}{\subset}

\newcommand{\goesto}{\rightarrow}

\newcommand{\imb}{\hookrightarrow}

\newcommand{\Sch}{\mathcal{S}}

\newcommand{\Hper}{H_{\text{per}}^\s}
\newcommand{\Hdper}{\dot{H}_{\text{per}}^\s}

\newcommand{\Sob}[2]{\lVert#1\rVert_{#2}}

\newcommand{\abs}[1]{\lvert#1\rvert}

\newcommand{\req}[1]{(\ref{#1})}

\newcommand{\FLp}{F_{L^p}}
\newcommand{\FHg}{F_{H^{-\gam/2}}}

\newcommand{\THs}{\Tht_{H^\s}}

\newtheorem{thm}{Theorem}

\newtheorem{defn}{Definition}
\newtheorem{prop}{Proposition}[section]

\newtheorem{lemm}[prop]{Lemma}
\newtheorem{lemmm}{Lemma}[prop]

\newtheorem{rmk}{Remark}[section]
\newtheorem{coro}{Corollary}[prop]

\newtheorem*{sett}{Standing Hypotheses}

\numberwithin{equation}{section}

\title[determining form for subcritical SQG]
{A determining form for the subcritical surface quasi-geostrophic equation}

\author{Michael S. Jolly$^{1}$$^\dagger$}
\address{$^{1}$Department of Mathematics\\
Indiana University\\ Bloomington, IN 47405}

\author{Vincent R. Martinez$^{2}$}
\address{$^{2}$Department of Mathematics\\
Tulane University\\ New Orleans, LA 70118}

\author{Tural Sadigov$^{3}$}
\address{$^3$ Department of Mathematics and Sciences\\
SUNY Polytechnic Institute\\
Utica, NY 13502}

\author{Edriss S. Titi$^{4}$}
\address{$^4$ Department of Mathematics\\
Texas A\&M University\\
College Station, TX 77843-3368\\
Also:
Department of Computer Science and Applied Mathematics\\
Weizmann Institute of Science\\
Rehovot, 76100, Israel}

\email[M.S. Jolly]{msjolly@indiana.edu}
\email[V. R. Martinez]{vmartin6@tulane.edu $\dagger$}
\email[T. Sadigov]{sadigot@sunyit.edu}
\email[E.S. Titi]{titi@math.tamu.edu, edriss.titi@weizmann.ac.il}

\thanks{$\dagger$ denotes corresponding author}

\begin{document}

\maketitle
\date{\today}

\centerline{\it  This paper is dedicated to the memory of Professor George R. Sell.}

\begin{abstract}
We construct a determining form for the surface quasi-geostrophic (SQG) equation with subcritical dissipation. In particular, we show that the global attractor for this equation can be embedded in the long-time dynamics of an ordinary differential equation (ODE) called a determining form. Indeed, there is a one-to-one correspondence between the trajectories in the global attractor of the SQG equation and the steady state solutions of the determining form.  
The determining form is a true ODE in the sense that its vector field is Lipschitz.  This is shown by combining De Giorgi techniques and elementary harmonic analysis.  Finally, we provide elementary proofs of the existence of time-periodic solutions, steady state solutions, as well as the existence of finitely many determining parameters for the SQG equation. 

\end{abstract}

\section{Introduction}
The three-dimensional (3D) quasi-geostrophic (QG) equation, from which the surface quasi-geostrophic (SQG) equation is derived, is a classical equation in geophysics used to describe the motion of stratified, rapidly rotating flows (cf. \cite{pedlosky}).  It asserts the conservation of potential vorticity of the flow, subject to dynamical boundary conditions, and represents the departure, to lowest order, of the dynamics of the flow from the so-called state of \textit{geostrophic balance}, in which the pressure gradient and the Coriolis force are in balance.  In the special case where the potential vorticity is identically zero, the dynamics of the QG equation reduces entirely to the evolution of its restriction to the two-dimensional boundary, which is the SQG equation.  Specifically, for $\gam\in(0,2)$, the dissipative SQG equation in non-dimensional variables is given by
\begin{align}\label{sqg}
			\bdy_t\tht+\kap\Lam^\gam\tht+u\cdotp{\del}\tht=f,\quad u=\Ri^\perp\tht, \quad \tht(x,0)=\tht_0(x),
	\end{align}
subject to periodic boundary conditions with fundamental periodic domain $\T^2=[-\pi,\pi]^2$.  Here, $\tht$ represents the scalar buoyancy or surface temperature of a fluid, which is advected along the two-dimensional velocity vector field, $u$, and $f$ is a given external source of heat.  
The velocity is related to $\tht$ by a Riesz transform, $\Ri^\perp:=(-R_2,R_1)$, where the symbol of $R_j$ is given by $-i\xi_j/\abs{\xi}$, for $j=1,2$.  For $0<\gam\leq2$, we denote by $\Lam^\gam:=(-\De)^{\gam/2}$ the operator whose symbol is $\abs{\xi}^\gam$, while its prefactor, $\kap$, is some fixed positive quantity 
and appears due to the physical Ekman pumping at the boundary due to mathematical analytical reasons (cf. \cite{desj:gren98}).  The equation is said to be \textit{subcritical} when $\gam>1$, \textit{critical} when $\gam=1$, and \textit{supercritical} when $\gam<1$.  In this article, we will consider the subcritical case $\gam>1$.  In this mathematical setting, we assume that the initial data, $\tht_0$, and $f$ are spatially periodic with fundamental domain, $\T^2$ and with average zero over $\T^2$.  Consequently, and since $\del\cdotp u=0$, it is easy to see that $\frac{d}{dt}\int_{\T^2}\tht(x,t)\ dx=0$, and hence that $\int_{\T^2}\tht(x,t)\ dx=0$, for all $t\geq0$.  We point out that throughout this paper we assume that $f(x)$ is time-independent, except when we discuss time-periodic solutions to \req{sqg} (see Appendix \ref{app}).

Since its introduction to the mathematical community, by Constantin-Majda-Tabak in \cite{cmt}, the SQG equation has been extensively studied. 
By now global well-posedness in various function spaces has been established and the issue of global regularity resolved in all, but the supercritical case (cf. \cite{caff:vass, const:vic, const:wu:qgwksol, ctv, cz:vic, hdong, kis:naz, kis:naz:vol, resnick}).  The long-time behavior and existence of a global attractor have also been studied in both the subcritical and critical cases (cf. \cite{carrillo:ferreira, chesk:dai1, const:coti:vic, coti:zelati, ctv, ju:qgattract}), as has the existence of steady states \cite{dai:steady}.  Moreover, the finite-dimensionality of the long-time behavior in the subcritical and critical SQG flows have been studied in \cite{chesk:dai:critqg, chesk:dai:subcritqg}, where the existence of \textit{determining modes} (see discussion below and Appendix \ref{app}) is established, and also in \cite{ctv, wang:tang}, where upper bounds for the fractal dimension of the global attractor of the critical and subcritical SQG equations are obtained.

The study of \textit{determining forms} was initiated in \cite{fkjt1, fkjt2} for the 2D NS equations (NSE).  A determining form is an ordinary differential equation (ODE) in an \textit{infinite-dimensional} Banach space of trajectories, which subsumes the dynamics of the original equation in a certain way.  A stronger expression of the finite-dimensionality of long-time behavior of dissipative evolutionary equations is the existence of an \textit{inertial manifold}, which is a finite-dimensional Lipschitz manifold that contains the global attractor of the original system and, moreover, attracts all solutions at an exponential rate (cf. \cite{cfnt, foias:sell:temam, foias:sell:titi} and references therein).  Restricted to the inertial manifold, the dynamics of the original system reduces to a finite-dimensional ODE, known as an \textit{inertial form}, in a \textit{finite-dimensional} phase space.  In fact, if an inertial manifold does exist, then knowledge of the low modes of the solution at a single point in time, enslaves the higher modes for \textit{all time}.  
However, while the existence of an inertial manifold has been established for a large class of dissipative equations, e.g., certain reaction-diffusions systems and the Kuramoto-Sivashinsky equation, it has been an outstanding open problem for the  2D NSE since the 1980s (cf. \cite{temam:dynsys}). 
It is as well open for the 1D damped, driven nonlinear Schr\"odinger (NLS), Korteweg de Vries equations (KdV), and 2D dissipative SQG equation.  Nevertheless, we show here that a determining form exists for the subcritical SQG equation.  

Currently, there are two ways to construct a determining form.  Determining forms of the ``first kind" encode as traveling waves,  projections of trajectories on the global attractor of the original system onto a sufficiently large, but fixed number of Fourier modes \cite{fkjt1}.  Alternatively, determining forms of the ``second kind," which is a more general approach, encode such trajectories as steady states \cite{fkjt2}.  In dissipative systems for which {determining modes} exist, the low Fourier modes of a {trajectory} on the global attractor, $\A$, suffice to characterize all higher modes of the trajectory.  More precisely, if two solutions on $\A$ agree up to a certain number of Fourier modes \textit{for all time}, then they must be the same solution, thus {defining a lifting, $W$, that maps the \textit{projection} of any trajectory in $\A$ to a unique trajectory in $\A$}.  An essential step in constructing a determining form is the extension of this lifting map to a certain Banach space.  We stress that the minimum number of such Fourier modes is \textit{independent} of the solution and depends only on the physical parameters of the equation.

Existence of a determining form of the second kind has also been established for the 1D damped, driven NLS equation \cite{jolly:sadigov:titi1} and damped, driven KdV equation \cite{jolly:sadigov:titi2}, both of which have determining modes (cf. \cite{goubet, jolly:sadigov:titi1, jolly:sadigov:titi2}).  Recently, it has been shown that the determining form of the second kind has the remarkable property that its solution can be parametrized by a \textit{single parameter}, whose evolution is governed by a \textit{one-dimensional} ODE, called the ``characteristic parametric determining form" (cf. \cite{fjlt}).  Indeed, the steady states of the characteristic parametric determining form can be used to characterize trajectories on the corresponding global attractor of the original equation.  {{}{This may lead to a novel}} approach to studying the geometry of the global attractor, a subject of ongoing research.

Finally, determining forms of the second kind provide a unified approach to determining forms in the sense that the modal projection can be replaced by a general interpolant operator depending on the type of determining parameters that exist for the original equation.  In the case of the 2D NSE, modes, nodes, and volume elements are all determining parameters (cf.  \cite{cjt1, cjt2, foias:prodi, foias:temam:nodes, jones:titi:nodes, jones:titi:volelts, jones:titi}), and as a result, each can induce determining forms of the second kind (cf. \cite{fkjt1}).  We refer also to \cite{holst:titi} for a notion of determining parameters for the 3D NSE.  This determining form was inspired by a ``feedback control" approach to data assimilation in \cite{azouani:titi, azouani:olson:titi}, where a defining model for a data assimilation algorithm was studied for the 1D Chaffee-Infante equation and 2D NSE, respectively. 
The defining model of the algorithm is a {{}{companion}} system to the original equation and is defined by inserting the collected observables of a reference solution into  the original model through a ``feedback control term" (see \req{sqg:da}).  The solution to this companion system synchronizes with the reference solution forward in time at an exponential rate.  The extension of the $W$ mapping follows a similar idea except that the initial time is taken backward to $-\infty$, as we will demonstrate here.

The main result of this paper, which establishes the existence of a determining form of the second kind for the subcritical SQG equation as well as various properties its solutions, is stated as Theorem \ref{thm:gen:detform}.  In particular, it is shown that there exist Banach spaces $X, Y$, a bounded set $B\sub X$, and Lipschitz vector field, $F:B\goesto X$, such that
	\begin{align}\label{df:ode}
		\frac{dv}{d\tau}=F(v),
	\end{align}
defines an ODE satisfying the following property: the steady states of \req{df:ode} that are contained in $B$ correspond in a one-to-one fashion to trajectories on the global attractor of \req{sqg}.  Moreover, there exists a subset $B'\sub B$ such that every solution of \req{df:ode} corresponding to $v_0\in B'$ must converge, as $\tau\goesto\infty$, to a steady state solution of \req{df:ode}.  In addition to ensuring other properties of \req{df:ode}, Theorem \ref{thm:gen:detform}, in fact, provides a general recipe for establishing a determining form for dissipative systems (cf. \cite{fkjt2, jolly:sadigov:titi1, jolly:sadigov:titi2}).  The crucial ingredients, namely the well-posedness theory of the underlying system and the Lipschitz property of the resulting map $W$ are established through Propositions \ref{da:exist} and \ref{W-Lip-2} in section \ref{sect:detform}; this is the heart of this paper.
The difficulties presented by the fractional dissipation in \req{sqg} are well-known and for the study undertaken here, amount to obtaining uniform $L^p$ estimates.  However, the $L^p$ bounds available for the companion equation depend on the number of modes required for their solutions to synchronize (see Proposition \ref{prop:lp}).  This renders these bounds unsuitable for showing that the evolution governed by \req{df} is well-defined and has a Lipschitz vector field (see section \ref{sect:detform}).  To overcome this obstacle, we employ De Giorgi-type estimates to bootstrap bounds from $L^2$ to $L^\infty$.  Coupled with elementary harmonic analysis techniques, we are able to show that the $L^\infty$ estimates thus obtained are indeed independent of the number of modes (see section \ref{sect:linfty}), ultimately allowing us to establish the properties required for existence.  These steps are detailed in section \ref{outline}, where we provide an outline of our approach.  Finally, in Appendix \ref{app}, we provide an elementary proof of existence of time-periodic solutions, steady state solutions, and finitely many determining parameters for \req{sqg}, which include as special cases determining modes and volume elements.  Our proof of the existence of steady state solutions complements the result in \cite{dai:steady}, where existence of such solutions is established when the domain is given as the whole plane, $\R^2$.  Our proof of the existence of finitely many determining parameters complements that of Cheskidov and Dai in \cite{chesk:dai:subcritqg} (see Remark \ref{app:rmk}), where a proof of existence of determining \textit{modes} is given using an approach inspired by the phenomenology of dissipation length scales from turbulence theory.  On the other hand, to the best of the authors' knowledge, the existence of time-periodic solutions to \req{sqg} with a time-periodic external source term appears to be new.

\section{Preliminaries}\label{sect:prelim}

\subsection{Function spaces: $L^p_{per}$, $V_\s$, $H_{{per}}^\s$, $\dot{H}_{{per}}^\s$ }\label{hs:def}
Let $1\leq p\leq\infty$, $\s\in \R$ and $\T^2=\R^2/(2\pi\Z)=[-\pi,\pi]^2$.
Let $\mathcal{M}$ denote the set of real-valued Lebesgue measureable
functions over $\T^2$.  Since we will be working with periodic functions,
define
	\begin{align}
		\mathcal{M}_{per}:=\{\phi\in\mathcal{M}:\phi(x,y)=\phi(x+2\pi,y)=\phi(x,y+2\pi)=\phi(x+2\pi,y+2\pi)\ \text{a.e.}\}.
	\end{align}
Let $C^\infty(\T^2)$ denote the class of functions which are infinitely
differentiable over $\T^2$.  Define $C^\infty_{per}(\T^2)$ by
	\begin{align}\notag
		C^\infty_{per}(\T^2):=C^\infty(\T^2)\cap\mathcal{M}_{per}.
	\end{align} 
For $1\leq p\leq\infty$, define the periodic Lebesgue spaces by
	\begin{align}\notag
		L^p_{per}(\T^2):=\{\phi\in\mathcal{M}_{per}: \Sob{\phi}{L^p}<\infty\},
	\end{align}
where
	\begin{align}\notag
		\Sob{\phi}{L^p}:=\left(\int_{\T^2} |\phi(x)|^p\ dx\right)^{1/p},\quad 1\leq p<\infty,\quad\text{and}\quad \Sob{\phi}{L^\infty}:=\ \stackrel[x\in\T^2]{}{\esssup}|\phi(x)|.
	\end{align}
Let us also define
	\begin{align}\label{Z:space}
		\mathcal{Z}:=\{\phi\in L^1_{per}:\int_{\T^2} \phi(x)\ dx=0\}.
	\end{align}

Let $\hat{\phi}(\mathbf{k})$ denote the Fourier coefficient of $\phi$ at wave-number $\mathbf{k}\in\Z^2$.  For any real number $\s\geq0$, define the homogeneous Sobolev space, $\Hdper(\T^2)$, by
	\begin{align}\label{Hdper}
		\Hdper(\T^2):=\{\phi\in L^2_{per}(\T^2)\cap\mathcal{Z}:\quad \Sob{\phi}{\dot{H}^\s}<\infty\},
	\end{align}
where
	\begin{align}\label{hdper:norm}
		{{} \Sob{\phi}{\dot{H}^\s}^{2}:=}\sum_{\mathbf{k}\in\Z^2\smod\{\mathbf{0}\}}\abs{\mathbf{k}}^{2\s}|\hat{\phi}(\mathbf{k})|^2.
	\end{align}
For $\s\geq0$, we define the inhomogeneous Sobolev space, $\Hper(\T^2)$, by
	\begin{align}\label{Hper}
		\Hper(\T^2):=\{\phi\in L^2_{per}(\T^2):\Sob{\phi}{H^\s}<\infty\},
	\end{align}
where
	\begin{align}\label{hper:norm}
		\Sob{\phi}{{H}^\s}^2:=\sum_{\mathbf{k}\in \Z^2}(1+\abs{\mathbf{k}}^2)^{\s}|\hat{\phi}(\mathbf{k})|^2.
	\end{align}

Let $\mathcal{V}_0\sub\mathcal{Z}$ denote the set of trigonometric polynomials with mean zero over $\T^2$ and set
	\begin{align}\label{hper}
		V_\s:=\overline{\mathcal{V}_0}^{H^\s},
	\end{align}
where the closure is taken with respect to the norm given by \req{hper:norm}.  
Observe that the mean-zero condition can be equivalently stated as $\hat{\phi}(\mathbf{0})=0$.  Thus, $\Sob{\cdotp}{\dot{H}^\s}$ and $\Sob{\cdotp}{H^\s}$ are equivalent as norms over $V_\s$.  Moreover, by  Plancherel's theorem we have
	\begin{align}\notag
		\Sob{\phi}{\dot{H}^\s}=\Sob{\Lam^\s \phi}{L^2}.
	\end{align}
Finally, for $\s\geq0$, we identify $V_{-\s}$ as the dual space, $(V_\s)'$, of $V_\s$, which can be characterized as the space of all bounded linear functionals, $\psi$, on $V_\s$ such that
	\[
		\Sob{\psi}{\dot{H}^{-\s}}<\infty.
	\]
Therefore, we have the following continuous embeddings
	\[
		V_\s\imb V_{\s'}\imb V_0\imb V_{-\s'}\imb V_{-\s},\quad 0\leq\s'\leq\s.
	\]

\begin{rmk}
Since we will be working over $V_\s$ and $\Sob{\cdotp}{\dot{H}^\s}$ and $\Sob{\cdotp}{H^\s}$ determine equivalent norms over $V_\s$, we will often denote $\Sob{\cdotp}{\dot{H}^\s}$ simply by $\Sob{\cdotp}{H^\s}$ for convenience. 
\end{rmk}

\subsection{Littlewood-Paley decomposition}\label{sect:LWP}
We define ``smooth spectral projection"  by the Littlewood-Paley decomposition.  We presently give a brief review of this decomposition.  More thorough treatments can be found in \cite{bcd, danchin:notes, runst:sick}.  We state the decomposition for ${\R^2}$, but point out that it is also valid in the case $\T^2$ for periodic distributions (see Remark \ref{periodic:rmk}).

Let $\psi_0$ be a smooth, radial bump function such that $\psi_0(\xi)=1$ when $[\abs{\xi}\leq 1/4]\sub{\R^2}$, and
	\begin{align}\label{psi:0}
		0\leq\psi_0\leq1\ \text{and}\ \spt\psi_0=\{\abs{\xi}\leq 1/2\}.
	\end{align}
Define ${\phi}_0(\xi):=\psi_0(\xi/2)-\psi_0(\xi)$.  Observe that
	\begin{align}\notag
		0\leq{\phi}_0\leq 1\ \text{and}\ \spt{\phi}_0=\{1/4\leq\abs{\xi}\leq1\}.
	\end{align}
Now for each integer $j\geq0$, define
	\begin{align}\label{psi:phi}
		{\phi}_j(\xi):={\phi}_0(\xi2^{-j}).
	\end{align}
Then, in view of the above definitions, we clearly have
	\begin{align}\label{lp:loc}
		\spt{\phi}_j=\{2^{j-2}\leq\abs{\xi}\leq2^{j}\}.
	\end{align}
If we let ${\phi}_{-1}:=\psi_0$ and ${\phi}_{j}\equiv0$ for $j<-1$, we observe that
	\begin{align}\label{pou}
		\sum_{j\in\Z}{\phi}_j(\xi)=1,\ \text{for }\ \xi\in{\R^2}.
	\end{align}
Let $g\in \Sch'(\R^2)$ be a tempered distribution over $\R^2$, then one can then define
	\begin{align}\label{lp:block:def}
		\lpk g:=h_k*g, \quad \til{\lp}_kg:=\sum_{|k-\ell|\leq 2}\lpl g,\quad S_kg:=\sum_{\ell\leq k}\lpl g,\quad T_k:=I-S_k,
	\end{align}
where $h_k:=\check{\phi}_k$ is the inverse Fourier transform of ${\phi}_k$.  We call the operators, $\lpk$, Littlewood-Paley projections.  For convenience, we will sometimes use the shorthand 
	\begin{align}\notag
		g_k:=\lpk g.
	\end{align}
 One can show that \req{pou} implies that
	\begin{align}\label{lp:dc}
		g=\sum_{j\geq-1}\lpj g\quad\text{for all}\ g\in \Sch'({\R^2}).
	\end{align}

For functions whose spectral support is compact, one has the Bernstein inequalities, which we will make frequent use of throughout the article.

\begin{prop}[Bernstein inequalities \cite{bcd}]\label{bern}
Let $1\leq p\leq q\leq\infty$ and $g\in\Sch'({\R^2})$, the dual space of the Schwartz space $\Sch(\R^2)$ .  There exists an absolute constant $C>0$, depending only on $\be, \phi_0, \psi_0$ such that for each $j\geq-1$ and $\be\in\R$, we have
	\begin{align}\notag
		\begin{split}
		&C^{-1}2^{j\be}\Sob{\lpj g}{L^q}\leq\Sob{\Lam^\be \lpj g}{L^q}\leq C2^{j(\be+2(1/p-1/q))}\Sob{\lpj g}{L^p}\\
				&\Sob{\Lam^\be S_jg}{L^q}\leq  C2^{j(\be+2(1/p-1/q))}\Sob{S_j g}{L^p}.
		\end{split}
	\end{align}
\end{prop}

\begin{rmk}\label{sob:to:bes}
Observe that the Bernstein inequalities provide a convenient characterization of Sobolev spaces, $\dot{H}^\be(\T^2)$, for any $\be\in\R$.  Indeed, for $g\in\dot{H}^{\be}$, we have
	\begin{align}\notag
		C^{-1}\sum_{j\geq0}2^{2j\be}\Sob{\lpj g}{L^2}^2\leq\Sob{g}{\dot{H}^\be}^2\leq C\sum_{j\geq0}2^{2j\be}\Sob{\lpj g}{L^2}^2,
	\end{align}
for some absolute constant, $C>0$, which depends only on $\be$.  This can easily be proved by applying the Plancherel theorem and using the fact that $g=\sum_{j\geq0}\lpj g$ for periodic, mean zero functions over $\T^2$.
\end{rmk}

We will also make crucial use of the Littlewood-Paley inequality (cf. \cite{graf, workman}).  To state it, we first define the Littlewood-Paley square function.
	\begin{align}\label{lp:sq}
		\Sq g:=\left(\sum_j(\lpj g)^2\right)^{1/2}\quad\text{and}\quad\til{\Sq} g:=\left(\sum_{j}(\lpjc g)^2\right)^{1/2}.
	\end{align}
We will invoke it in the following form.

\begin{prop}[Littlewood-Paley Inequality]\label{prop:lpthm}
Let $1<p<\infty$.  Then $\Sq$ is a bounded operator in $L^p(\R^2)$ such that
	\begin{align}\label{lpthm1}
		\Sob{\Sq g}{L^p}\leq C_p\Sob{g}{L^p},
	\end{align}
where $C_p:=C'\max\{p,(p-1)^{-1}\}$, for some absolute constant $C'>0$, independent of $p$.
\end{prop}

\begin{rmk}By \req{lp:block:def} and the Minkowski inequality, note that the analogous statement for $\til{\Sq}$ holds as well.
\end{rmk}

\begin{rmk}\label{periodic:rmk}
Consider the same bump functions $\{\phi_j\}_{j\in\Z}$ as specified above, which satisfied \req{psi:0}-\req{pou}.  The Bernstein inequality (Proposition \ref{bern}) and Littlewood-Paley inequality (Proposition \ref{prop:lpthm})  also hold in the case of functions, $g$, which are periodic over $\T^d$, provided that we replace the Littlewood-Paley operators by
	\begin{align}
		\lpj^{per}g(x)=\frac{1}{(2\pi)^d}\int_{\T^d}\ph^{per}_j(y)g(x-y)\ dy,\quad\text{where}\quad  \ph^{per}_j(x)=\sum_{\mathbf{k}\in\Z^d}\phi_j(\mathbf{k})e^{i\mathbf{k}\cdotp x}.
	\end{align}
We refer to \cite{bcd, danchin:notes} for the Bernstein inequality and to \cite{workman} for a thorough treatment of classical harmonic analysis operators in the periodic setting, which includes the Littlewood-Paley square function.
\end{rmk}

\subsection{Inequalities for fractional derivatives}

We will make use of the following bound for the fractional Laplacian, which can be found for instance in \cite{const:gh:vic, ctv, ju:qgattract}.

\begin{prop}\label{lb}
Let $p\geq2$, $0\leq\gam\leq2$, and $g\in C_{per}^\infty(\T^2)$. Then 
	\begin{align}\notag
		\int_{\T^2}|g|^{p-2}(x)g(x)(\Lam^\gam g)(x)\ dx\geq \frac{2}p\Sob{\Lam^{\gam/2}(|g|^{p/2})}{L^2}^2,
	\end{align}
If additionally $p$ is an even integer and $\phi\in\mathcal{Z}$, then
	\begin{align}\notag
		\int_{\T^2}g^{p-1}(x)(\Lam^\gam g)(x)\ dx\geq C\Sob{g}{L^p}^p+\frac{1}p\Sob{\Lam^{\gam/2}(g^{p/2})}{L^2}^2,
	\end{align}
holds, for some $C=C(\gam)$, independent of $p$.
\end{prop}

Later, we will derive a level-set inequality in the spirit of De Giorgi (cf. \cite{caff:vass}).  For this, we will make use of the following fact from \cite{chesk:dai1}.

\begin{prop}\label{lem:lvlsetpos}
Let $\lam>0$ and $g\in C^\infty_{per}(\T^2)$.  Then for $(g-\lam)_+=\max\{g-\lam,0\}$, we have
	\begin{align}\notag
		(\Lam^\gam g)(g-\lam)_+\geq (g-\lam)_+(\Lam^\gam(g-\lam)_+).
	\end{align}
\end{prop}

We will  also make use of the following calculus inequality for fractional derivatives (cf. \cite{ctv} and references therein):

\begin{prop}\label{prod:rule}
Let $g,h\in C_{per}^\infty(\T^2)$, $\be>0$, $p\in(1,\infty)$.  For $1/p=1/p_1+1/p_2=1/p_3+1/p_4$, and $p_2,p_3\in(1,\infty)$, there exists an absolute constant $C>0$ that depends only on $\be,p,p_i$ such that
	\begin{align}\notag
		\Sob{\Lam^\be(gh)}{L^p}\leq C\Sob{h}{L^{p_1}}\Sob{\Lam^\be g}{L^{p_2}}+C\Sob{\Lam^\be h}{L^{p_3}}\Sob{g}{L^{p_4}}.
	\end{align}
\end{prop}

Finally, we will frequently apply the following interpolation inequality, which is a special case of the Gagliardo-Nirenberg interpolation inequality and can be proven with Plancherel's theorem and the Cauchy-Schwarz inequality.

\begin{prop}\label{interpol}
Let $g\in V_\be$ and  $0\leq \al\leq \be$.
Then there exists an absolute constant $C>0$ that depends only on $\al,\be$ such that
	\begin{align}\label{gn:ineq}
		\Sob{\Lam^{\al} g}{L^2}\leq C\Sob{\Lam^{\be}g}{L^2}^{\frac{\al}{\be}}\Sob{g}{L^2}^{1-\frac{\al}{\be}}.
	\end{align}
\end{prop}

\subsection{Maximum principle and Global Attractor of SQG equation}\label{sect:ball}
Let us recall the following estimates for the reference solution $\tht$ (cf. \cite{ctv, ju:qgattract, resnick}).

\begin{prop}\label{prop:sqg:ball}
Let $\gam\in(0,2]$ and $\tht_0, f\in {L}^p_{per}(\T^2)\cap\mathcal{Z}$. 
Suppose that $\tht\in {L}_{per}^p(\T^2)$ is a smooth solution of \req{sqg} such that $\tht(\cdotp, 0)=\tht_0(\cdotp)$.  There exists an absolute constant $C>0$ such that for any $p\geq2$, we have
	\begin{align}\label{fp}
		\Sob{\tht(t)}{L^p}\leq \left(\Sob{\tht_0}{L^p}-\frac{1}C{\FLp}\right)e^{-C{{{\kap}}}t}+\frac{1}{C}{\FLp},\quad {\FLp}:=\frac{1}{{\kap}}\Sob{f}{L^p}.
	\end{align}
Moreover, for $p=2$ and $f\in V_{-\gam/2}$, we have
	\begin{align}\label{fgam}
		\Sob{\tht(t)}{L^2}^2\leq \left(\Sob{\tht_0}{L^2}^2-{\FHg^2}\right)e^{-{{{\kap}}} t}+{\FHg^2}, \quad \FHg:=\frac{1}{{\kap}}\Sob{f}{H^{-\gam/2}}.
	\end{align}
\end{prop}

{
It was shown in \cite{ju:qgattract} for the subcritical
range $1<\gam\leq2$, that equation \req{sqg} has an absorbing ball in
$V_\s$
and corresponding global attractor
$\A\sub V_\s$ when $\s>2-\gam$.
In other words, there is
a bounded set $\mathcal{B}\sub V_\s$
characterized by the property that for any bounded set $E\sub V_\s$,
there exists $t_0=t_0(E)>0$ such that
$S(t)E\sub\mathcal{B}$ for all $t\geq t_0$.
Here $\{S(t)\}_{t\geq0}$
denotes the semigroup of the corresponding dissipative equation.
}

\begin{prop}[Global attractor]\label{prop:ga}
Suppose that $1<\gam\leq2$ and  $\s>2-\gam$.  Let $f\in V_{\s-\gam/2}\cap {L}_{per}^p(\T^2)$, where $1-\s<2/p<\gam-1$.  Then \req{sqg} has an absorbing ball $\mathcal{B}_{H^\s}$ given by
	\begin{align}\label{sqg:hs:ball}
		\mathcal{B}_{H^\s}:=\{\tht\in V_\s: \Sob{\tht}{H^\s}\leq \THs\},
	\end{align}
for some absolute constant $\THs=\THs(f,\kap,\gam,\s,\be)<\infty$.  Also, the solution operator $S(t)\tht_0=\tht(t)$, $t>0$ of \req{sqg} defines a semigroup in the space $V_\s$ such that $S(t)$ is continuous in $H^\s$ for each $t>0$, and $S:[0,t]\goesto H^\s$ is continuous for each $\tht_0\in H^\s$ fixed.  Moreover, \req{sqg} possesses a global attractor $\A\sub V_\s$, i.e., $\A$ is a compact, connected subset of $V_\s$ satisfying the following properties:
	\begin{enumerate}
		\item $\A$ is the maximal bounded invariant set, i.e., $S(t)\A=\A$ for all $t\geq0$, and hence $S(t)\A=\A$ for all $t\in\R$.;
		\item  $\A$ attracts all bounded subsets in $V_\s$ in the topology of $\dot{H}_{{per}}^\s$.
	\end{enumerate}
\end{prop}

Before we move on to the a priori analysis, we will set forth the following convention for constants.
\begin{rmk}
In the estimates that follow below, $c, C$, will denote generic positive absolute constants, which depend only on other non-dimensional scalar quantities, and may change line-to-line in the estimates.   We also use the notation $A\lesssim B$ and $A\sim B$ to denote the relations $A\leq cB$ and $c'B\leq A\leq c''B$, respectively, for some absolute constants $c, c', c''>0$.
\end{rmk}

\section{The determining form}\label{outline}
Let $\s>2-\gam$, $P_N$ denote orthogonal projection onto the wavenumbers $|\ell|\leq N$ and $S_m$ denote the Littlewood-Paley projection defined as in \req{lp:block:def}.  For $m\geq2$, we define the Banach spaces
{
	\begin{align}\label{X}
		\begin{split}
		 X&:={C}_{b}^1(\mathbb{R}; S_{m-1}V_\s)= \{ v:\mathbb{R} \rightarrow  S_{m-1}V_\s: \text{differentiable, } \Sob{v}{X}<\infty \}, \\
		Y&:={C}_{b}(\mathbb{R}; V_\s)= \{w: \mathbb{R}\rightarrow V_\s: \text{continuous, } \Sob{w}{Y}<\infty \},
		\end{split}
	\end{align}
}
equipped with the following norms,
	\begin{align}\label{XY:norms}
		\Sob{v}{X}= \sup_{s\in \mathbb{R}} \Sob{v(s)}{H^{\s}}+\frac{1}{2^{(2+\s)m}}\sup_{s\in\mathbb{R}}\Sob{v'(s)}{H^\s}, \qquad \Sob{w}{Y}= \sup_{s\in \mathbb{R}} \Sob{w(s)}{H^{\s}},
	\end{align}	
where $'$ denotes $d/ds$.

Our first concern is the following problem: Let $1<p<\infty$ such that $1/p+1/2<\gam/2$.  Given $f\in V_{\s-\gam/2}\cap L^p(\T^2)$ and $\rho>0$, find $\mu>0$ large enough, depending on $\rho, f, \kap$, and find $m\in \mathbb{N}$ large enough, such that for all $v\in B_{X}^{\rho}(0):=\{v\in X: \|v\|_{X}<\rho\}$,  the equation
	\begin{align}\label{sqg:da}
		\begin{split}
		&\bdy_sw+\kap\Lam^\gam w+\til{u}\cdotp\del w=f-\mu S_m(w-v),\quad\til{u}=\Ri^\perp w,
		\end{split}
	\end{align}
has a unique solution $w\in Y$.  We will choose $\rho>0$ according to the radius of the global attractor, $\A$ of \req{sqg}.  
{{}
Indeed, we have the following.

\begin{prop}\label{prop:X:att}
Let $F_{H^{-\gam/2}}$ be given by \req{fgam} and define
	\begin{align}\label{R:rad}
		R:=C\left(1+\kap\right)\left(1+F_{H^{-\gam/2}}\right)^2.
	\end{align}
Then $S_{m-1}\mathcal{A}\sub B_X^{4R}(0)$.
\end{prop}
\begin{proof}
Let $\tht(\cdotp)\sub\mathcal{A}$.  Then Proposition \ref{prop:sqg:ball} implies that $\sup_{t\in\R}\Sob{\tht(t)}{L^2}\leq F_{H^{-\gam/2}}$.  Observe that by Proposition \ref{prop:ga} and the definition \req{XY:norms}, it suffices to estimate $S_{m-1}\bdy_t\tht$.  We have
	\[
		\bdy_t\tht=-\kap\Lam^\gam\tht {{}{-}}\del\cdotp(u\tht)+f.
	\]
Since $S_{m-1}\mathcal{A}\sub V_\s$, by Remark \ref{sob:to:bes}, it suffices to estimate $2^{j\s}\Sob{\lpj\bdy_t\tht}{L^2}$ for $0\leq j\leq m-1$.  Applying Proposition \ref{bern} we obtain
	\begin{align}
		2^{j\s}\Sob{\lpj\Lam^{\gam}\tht}{L^2}&\leq C2^{j(\gam+\s)}\Sob{\lpj\tht}{L^2},\notag\\
		2^{j\s}\Sob{\lpj\del\cdotp(u\tht)}{L^2}&\leq C2^{j(\s+1+2(1-1/2))}\Sob{\lpj(u\tht)}{L^1}\leq 2^{j(\s+2)}\Sob{\tht}{L^2}^2,\notag\\
		2^{j\s}\Sob{\lpj f}{L^2}&\leq C2^{j(\s {{} +} \gam/2)}\Sob{\lpj f}{H^{-\gam/2}},
	\end{align}
It follows that
	\begin{align}
		\Sob{S_{m-1}\bdy_t\tht}{H^\s}^2&\leq C\left(2^{2m(\gam+\s)}\Sob{\tht}{L^2}^2+2^{2m(\s+2)}\Sob{\tht}{L^2}^4+C(m)\Sob{f}{H^{-\gam/2}}^2\right),\notag\\
			&\leq C\kap^2\left(2^{2m(\gam+\s)}\frac{F_{H^{-\gam/2}}^2}{\kap^2}+2^{2m(\s+2)}\frac{F_{H^{-\gam/2}}^4}{\kap^2}+C(m)F_{H^{-\gam/2}}^2\right),
	\end{align}
where $C(m)\leq 2^{2m(\s {{} +} \gam/2)}$ if $\s>\gam/2$, and $C(m)\leq (1- {{} 2} ^{\gam{{} +}2\s})^{-1}$, otherwise.  Therefore
	\begin{align}
		2^{-(2+\s)m}\Sob{\bdy_tS_{m-1}\tht}{H^\s}\leq R,
	\end{align}
so that $\Sob{\tht}{X}\leq 4R$.
\end{proof}
}

The unique solution, $w\in Y$, of \req{sqg:da} that we will establish, will then define a well-defined map $W:B_X^\rho(0)\goesto Y$, $v\mapsto w$, $w=W(v)$, provided that $\mu$ and $m$ are large enough.  Moreover, given any steady state solution  $\tht^*\in\mathcal{A}$ of \req{sqg} (cf. \cite{dai:steady, fjlt}), $W$ induces the following evolution equation
	\begin{align}\label{df}
		\begin{split}
		\frac{dv(s)}{d\tau}(\tau)= -\Sob{v(\hspace{1.75pt}\cdotp)(\tau)-S_mW(v(\hspace{1.75pt}\cdotp)(\tau))}{X}^2 (v(s)(\tau)-S_m\tht^*),\quad v(0)=v_0\in\mathcal{B}_X^{\rho}(0).
		\end{split}
	\end{align}
If $S_mW:B_X^\rho(0)\goesto Y$ is a Lipschitz map, then \req{df} is an \textit{ordinary differential equation} in the Banach space $X$. 
This is formalized in the following theorem.

\begin{thm}\label{thm:gen:detform}
Let $\rho=4R$, where $R$ is given as in Proposition \ref{prop:X:att}, and let $\tht^*\in\mathcal{\A}$ be a steady state solution of \req{sqg}. Suppose that the Standing Hypotheses $(H1)$-$(H7)$, below, hold. Then the following are true.
	\begin{enumerate}[(i)]
		\item{The vector field in \req{df} is a Lipschitz map from the ball $\mathcal{B}_X^{\rho}(0)$ into $X$. Thus, \req{df}, is an ODE in $\mathcal{B}_{X}^{\rho}(0)$, and has short time existence and uniqueness. }
		\item{The ball $\mathcal{B}_X^{3R}(S_m\tht^*)= \{v\in X: \|v-S_m\tht^*\|_X< 3R\} \subset \mathcal{B}_X^{\rho}(0)$ is forward invariant in time, under the dynamics of the determining form \req{df}. Consequently, \req{df} has global existence and uniqueness for all initial data in $\mathcal{B}_X^{3R}(S_m\tht^*)$.}
		\item{Every solution of  the determining form \req{df}, with initial data $v_{0}\in\mathcal{B}_X^{3R}(S_m\tht^*)$, converges to a steady state solution of the determining form \req{df}. }
		\item{All of the steady state solutions of the determining form, \req{df}, that are contained in the ball $\mathcal{B}_X^{\rho}(0)$ are given by $v(s)=S_m\tht(s)$, for all $s\in \mathbb{R}$, where $\tht(s)$ is a trajectory that lies on the global attractor, $\mathcal{A}$, of \req{sqg}.}
	\end{enumerate}
\end{thm}
This theorem has been established for other equations in previous determining form papers \cite{fkjt2, jolly:sadigov:titi1, jolly:sadigov:titi2}.  
{ 
In particular, observe that $B_X^{3R}(S_m\tht^*)\subset B_X^{4R}(0)$ since $\bdy_tS_m\tht^*=0$.  On the other hand, property $(iv)$, above, is implied by the fact that $v(s)=S_mW(v(s))$, for $s\in\R$, only when 
	\[
		S_mw=v=S_mv,
	\]
since $v(s)\in S_{m-1}V_\s$, for each $s\in\R$, so that by \req{sqg:da}, $w$ must be a trajectory that satisfies \req{sqg}.  Thus, to establish Theorem \ref{thm:gen:detform} in our case, it suffices to show:

\begin{enumerate}[1.)]
	\item $W: B_X^{\rho}(0)\goesto Y$ exists and is well-defined;
	\item $S_mW: B_X^{\rho}(0)\goesto Y$ is Lipschitz, for some $m$.
\end{enumerate}
}

We demonstrate these {{}{claims}} in the next section.

\begin{rmk}
It is shown in \cite{fjlt} that \req{df} can be modified in order to improve an algebraic convergence rate of $v(\tau)$, as $\tau\goesto+\infty$, to the projection of some trajectory in the global attractor.  Replacing the power $2$ by the power $1$ on the $X$-norm in \req{df} yields a faster algebraic rate, while with further modification, one obtains exponential convergence as $\tau\goesto+\infty$.  
\end{rmk}

\section{Proof of existence of the determining form}\label{sect:detform}

We operate under the following conditions throughout both sections \ref{sect:detform} and \ref{sect:apriori}.
\begin{sett}
Assume the following
\begin{enumerate}[(H1)]
	\item $1<\gam<2$;
	\item $\s>2-\gam$;
	\item $p\in[1,\infty]$ such that $1-\s<2/p<\gam-1$, fixed;
	\item $f\in V_{\s-\gam/2}\cap L^{{p}}$, is time-independent;
	\item $v\in B_X^{4R}(0)$, {{}{where $R$ is given in Proposition \ref{prop:X:att} above;}}
	\item $\mu$ is large enough, i.e., satisfies;
		\begin{align} \label{MainCondition}
			\mu\geq c_0\kap\left(\frac{G_{L^p}}{\kap}\right)^{\gam/(\gam-1-2/p)},
		\end{align}
for some sufficiently large absolute constant $c_0$, and where $G_{L^p}$ is given by \req{Gp} below;
	\item $m$ is large enough, i.e., satisfies
		\begin{align}\label{modes:crucial}
		2^{\gam m}\geq c_0'\frac{\mu}{\kap}\quad\text{and}\quad	2^{m}\geq c_0''\left(\frac{G_{\s,\infty}}{G_{L^p}}\right)^{1/(1-2/p-\s)},
		\end{align}
for some sufficiently large absolute constants $c_0', c_0''>0$, and where $G_{\s,\infty}$ is given by \req{gsigtil} below.
\end{enumerate}
\end{sett}

We note that hypotheses $(H1)$ is precisely the subcritical range of dissipation for \eqref{sqg:da}. Whenever we talk about $H^{\sigma}$ or $L^{p}$ bounds, the parameters $\sigma$ and $p$ satisfy hypotheses $(H2)$ and $(H3)$, respectively. Moreover, the fixed trajectory $v\in X$ satisfies the hypothesis $(H5)$ by definition of $X$, which is also built into system \eqref{sqg:da}. Recall that to guarantee the existence of a determining form for \req{sqg}, it suffices to show that:  1) $W:B_X^{4R}(0)\goesto Y$ exists, 2) $W$ is well-defined, and 3)  $S_mW:B_X^{4R}(0)\goesto Y$ is Lipschitz.  These three objectives thus constitute the main results of this section.


\subsection{Existence of the $W$ map}

We first prove that the $W: B_X^{4R}(0)\goesto Y$ map exists and is well-defined.
To prove this, we must establish existence of solutions to \req{sqg:da}.

\begin{prop}[Existence]\label{da:exist}
Under the Standing Hypotheses, $(H1)-(H7)$, for each $v\in B_X^{4R}(0)$, there exists $w\in Y$ satisfying \req{sqg:da} with $\bdy_sw\in L^\infty(\R;V_{\s-\gam})\cap L^2_{loc}(\R;V_{\s-\gam/2})$.
\end{prop}

{{}
\begin{proof}[Sketch of proof]

Let $\eps\in(0,1)$ and consider the following parabolic regularization of \eqref{sqg:da}
	\begin{align}\label{sqg:daPar}
		\begin{split}
		&\bdy_sw^\eps-\eps \Delta w^\eps+\kap\Lam^\gam w^\eps+\til{u}^\eps\cdotp\del w^\eps=f^\eps-\mu S_m(w^\eps-v),\quad \til{u}^\eps=\Ri^\perp w^\eps,
		\end{split}
	\end{align}
where $f^\eps=\phi_\eps*f$ and $\phi_\eps(x):=\eps^{-2}\phi(x/\eps)$ for some smooth, compactly supported mollifier function $\phi$.
We will first establish existence of a global strong solution, $w^\eps\in L^\infty(\R;V_2)$, to \req{sqg:daPar}.  We will then conclude the proof by passing to the limit $\eps\goesto0$ to obtain existence of a global {{} solution to \req{sqg:da} in the sense of distribution.}

\subsubsection*{Step 1: Global well-posedness for truncations of \req{sqg:daPar}} 
Let $v\in B_X^{4R}(0)$, and let $N\in\N$ such that $N\geq 2^{m+1}$.  Let $P_N$ denote orthogonal projection onto  wavenumbers $|\ell|\leq N$.  Let $H_N:=P_NV_0$ denote the finite-dimensional subspace of $V_0$ spanned by $\{e^{ij\cdotp x}\}_{|j|\leq N}$.  Let $f_{N}:=P_Nf$ and $v_N=P_Nv$.  Given $k>0$, consider the initial value problem for the following system:
	\begin{align}\label{sqg:daParGal}
		\begin{split}
		&\bdy_sw_{k,N}^{\eps}-\eps \Delta w^\eps_{{}{k,N}}+\kap\Lam^\gam w_{{{}{k,N}}}^{\eps}+P_{N}(\til{u}_{k,N}^\eps\cdotp\del w_{k,N}^{\eps})=f_{N}^\eps-\mu S_m(w_{k,N}^{\eps}-v_N),\\
		&\til{u}_{k,N}^{\eps}=\Ri^\perp w_{k,N}^{\eps},\quad w_{k,N}^{\eps}{\big|}_{s=-k}=0.
		\end{split}
	\end{align}
Note that $w_{k,N}^\eps, f_{N}^\eps\in H_{N}$ and also, $S_mv_N\in H_N$ since $S_mv_N=S_mP_Nv_N=P_NS_mv_N,$ {{} for $N\geq 2^{m+1}$}.  
Since \req{sqg:daParGal} is equivalent to a system of ODEs with a {{}{locally}} Lipschitz vector field, it has a unique solution $w_{k,N}^\eps$ on some interval $[-k, S^*)$, for some $S^*>-k$.  Without loss of generality, we may assume that $[-k,S^*)$ is the maximal interval of existence and uniqueness of \req{sqg:daParGal}.  Note that strictly speaking $S^*$ depends on $k$ and $N$, i.e., $S^*=S^*_{k,N}$.  Our goal is to show that $S^*=\infty$.  To establish this, it is enough to show in this case that $\sup_{-k\leq s<S^*}\Sob{w^{\eps}_{k,N}(s)}{V_0}<\infty$.

We assume, by contradiction, that $S^*<\infty$, and then let us focus on the maximal interval of existence $[-k,S^*)$.  We first establish a bound on $\Sob{w^{\eps}_{k,N}(s)}{V_0}$, which is independent of $\eps, k, N, m$, and $S^*$ (cf. Proposition 6 \cite{jmt:sqgda}).  This will imply, among other things, that $S^*=\infty$.  Next, we establish estimates for $\Sob{w^\eps_{k,N}(s)}{V_0}$, for $s\in[-k,\infty)$, which depend on $\eps, k, m$, and grow in $s$, 
but are nevertheless, independent of $N$.  Indeed, one can show that for each ${{}{s^*}}>-k$ (cf. \cite{temamnse, temam:dynsys}): 
\begin{enumerate}[$(1_{\eps,k}$)]\addtocounter{enumi}{-1}
	\item {{}{$w_{k,N}^{\eps}\in L^\infty([-k,{{{}{\infty}}});V_0)$;}}
	\item $w_{k,N}^{\eps}\in L^2([-k,{{}{s^*}});V_1)$ {{}{and $\frac{d}{ds}w_{k,N}^{\eps}\in L^2([-k,{{}{s^*}});V_{-1})$;}}
	\item $w_{k,N}^{\eps}\in L^\infty([-k,{{}{s^*}});V_1)$;
	\item $\frac{d}{ds}w^\veps_{k,N}\in L^\infty([-k,{{}{s^*}});V_0)$;
	\item $w_{k,N}^{\eps}\in L^\infty([-k,{{}{s^*}});V_2)$.
\end{enumerate}
In particular, given any ${{}{s^*}}>0$, the corresponding estimates, above, guarantee that \req{sqg:daParGal} has a unique solution over $[-k,{{}{s^*}})$, but with bounds ultimately depending on $\eps, k, {{}{s^*}}$ and $m$.  Thus, for each $\eps, k, {{}{s^*}}$, we may {{}{use $(1_{\eps,k})$, the Rellich compactness theorem, and then {{}{the}} Aubin-Lions lemma}} (cf. \cite{const:foias, temamnse}) to extract a subsequence of $N$, denoted again by $w_{k,N}^\epsilon$, pass to the limit as $N\goesto\infty$ as in \cite{const:foias} to obtain a weak solution $w_k^\eps\in L^\infty([-k,{{}{s^*}});V_0)\cap L^2([-k,{{}{s^*}});V_1)$ to
	\begin{align}\label{sqg:daPar:k}
		\bdy_sw^\eps_k-\eps \Delta w^\eps_k+\kap\Lam^\gam w^\eps_k+\til{u}_k^{\eps}\cdotp\del w^\eps_k=f^\eps-\mu S_m(w^\eps_k-v),\quad \til{u}_k^\eps=\Ri^\perp w^\eps_k,\quad w^\eps_k{\big|}_{s=-k}=0.
	\end{align}
Since $w_{k,N}^{\eps}\in L^\infty([-k,{{}{s^*}});V_2)$ with a corresponding bound that is uniform in $N$,  we in fact have that $w^\eps_k\in L^\infty([-k,{{}{s^*}});V_2)$, so that $w^\eps_k$ is the \textit{unique} strong solution to \req{sqg:daPar:k} over $[-k,{{}{s^*}})$.  Since this is true for each ${{}{s^*}}>0$, we have that $w^\eps_k\in L_{loc}^\infty([-k,\infty),V_2)$, for each $\eps, k$.

Next, we establish estimates that are \text{independent of} $k$ and ${{}{s^*}}$, but \textit{depend on} $\eps, m$.  Indeed, let $p>1$ satisfy $(H1)$, then we bootstrap as follows:
	\begin{enumerate}[$(1_{\eps})$]
		\item $w^\eps_k\in L^\infty([-k,\infty);L^p)$;
		\item $w^\eps_k\in L^\infty([-k,\infty);V_1)$;
		\item $w^\eps_k\in L^\infty([-k,\infty);V_2)$,
	\end{enumerate}
with corresponding bounds that are independent of $k, {{}{s^*}}$.  The estimates that imply $(1_{\eps}), (2_\eps)$ can be found in {{} Propositions  7}, 9, respectively, of \cite{jmt:sqgda}, while the estimates that imply $(3_\eps)$ can be performed in an entirely similar spirit.  We omit the details to avoid repetition of argument.  

\subsubsection*{Step 2: Existence of solutions to \req{sqg:daPar}}
Let $w^{\eps,\ell}_k:=w^\eps_{k}{\big|}_{[-\ell,\ell]}$ for $\ell>0$.  
Since $V_2\imb V_\tau\imb V_0$ compactly and continuously, for any $0<\tau<2$, and we have that the family, $\{w^{\eps,\ell}_k\}_{\eps,\ell,k}$, satisfies $(3_\eps)$ and $(4_\eps)$, it follows from the Aubin-Lions lemma that there exists a subsequence $(k_1(j))_{j>0}$ such that $w^{\eps,1}_{k_1(j)}\goesto w^{\eps,1}$ as $j\goesto\infty$, for some $w^{\eps,1}\in C([-1,1];V_{3/2})$, which satisfies \req{sqg:daPar}.  Proceeding, inductively in the same manner, for each $\ell>1$, there exists a subsequence $\{k_\ell(j)\}_{j>0}\sub\{k_{\ell-1}(j)\}_{j>0}$ such that $w^{\eps,\ell}_{k_\ell(j)}\goesto w^{\eps,\ell}$, for some $w^{\eps,\ell}\in C([-\ell,\ell];V_{3/2})$ satisfying \req{sqg:daPar}.  Now consider the sequence given by $w^{\eps,\ell}_{k_\ell(\ell)}$.  Then $w^{\eps,\ell}_{k_\ell(\ell)}\goesto w^{\eps}$, as $\ell\goesto\infty$, for some $w^{\eps}\in C(\R;V_{3/2})$.  Since $w^\eps{\big|}_{[-\ell,\ell]}=w^{\eps,\ell}$ and $(4_\eps)$ is satisfied for each $w^{\eps,\ell}$ uniformly in $\ell$, we may deduce that $w^\eps\in L^\infty(\R; V_2)$ and satisfies \req{sqg:daPar}.

\subsubsection*{Step 3: Passage to the limit $\eps\goesto0$} 


We establish bounds for $w^\eps$, which depend on $m$, but are \textit{independent of} $\eps$.  First, we observe that $w^{\eps}\in L^\infty(\R;L^p)$ uniformly in $\eps$, due to $(1_\eps)$.  We then establish the following with corresponding $\eps$-independent bounds:
\begin{enumerate}[(1)]
	\item $w^\eps\in L^\infty(\R;V_\s)\cap L_{loc}^2(\R; V_{\s+\gam/2})$;
	\item $\bdy_sw^\eps\in L^\infty(\R; V_{\s-2})\cap L^2_{loc}(\R;V_{\s+\gam/2-2})$\;.
\end{enumerate}
We then consider the family $\{w^{\eps,M}\}_{\eps,M}$, where $w^{\eps, M}=w^\eps{\big|}_{[-M,M]}$, and  argue similar to Step 2, upon combining the Aubin-Lions lemma with a Cantor diagonal argument, to deduce the existence of a subsequence $(\eps_j)_{j>0}$ such that $\eps_j\goesto0$ and $w^{\eps_{j}}\goesto w^{\eps}$, as $j\goesto\infty$, for some $w\in C(\mathbb{R}; V_{\til{\s}})$, for some $\til{\s}<\s$.  Observe that necessarily we have $w\in L^{\infty}(\mathbb{R}; V_\s)\cap L^{2}(\mathbb{R};V_{\s+\gam/2})$.  Furthermore, we have
	\begin{align}\notag
	\begin{split}
		 w^{\eps_{j}} &\rightharpoonup  w\quad\quad \text{ weak-}\star \text{ in } L^{\infty}(\mathbb{R}; V_\s),\\
		w^{\eps_{j}} &\rightarrow  w\quad  \text{strongly } \text{ in } L^{2}_{loc}(\mathbb{R}; V_{\til{\s}}),\quad \text{ for any } \til{\s}<\s,\\
		\partial_{s} w^{\eps_{j}} &\rightharpoonup  \partial_{s}w\quad \text{ weak-}\star \text{ in } L^{\infty}(\mathbb{R}, V_{\s-2}),
	\end{split}
	\end{align}
as $\eps_{j}\rightarrow 0$.  We can therefore pass to the limit, in the sense of distribution, {{}  using the Banach-Alaoglu theorem} and show that $w$ satisfies \req{sqg:da}.  Lastly, since
	\begin{align}\notag
		\partial_{s}w=-\kap\Lam^\gam w-\til{u}\cdotp\del w+f-\mu S_m(w-v),\quad \til{u}=\Ri^\perp w,
	\end{align}
we have that $\partial_{s}w\in L^{\infty}(\mathbb{R}; V_{\s-\gamma})\cap L^{2}_{loc}(\mathbb{R}; V_{\s-{\gamma}/{2}})$, and the above equation holds in $L^{\infty}(\mathbb{R};V_{\s-\gamma})\cap L^{2}_{loc}(\mathbb{R}; V_{\s-{\gamma}/{2}})$. This completes the proof.
\end{proof}

}

\begin{prop}[Uniqueness]\label{Uniqueness}
Assume $(H1)-(H7)$.  There exists a unique bounded solution $w\in Y$ of \req{sqg:da}, where
$X$ and $Y$ are defined as in \req{X}.
\end{prop}

\begin{proof}
Suppose there are two bounded solutions of \req{sqg:da}, $w_{1}$ and $w_{2}$, in $Y$ corresponding to the same $v\in B_{X}^{4R}(0)$.  Then $w_1, w_2$ satisfy
\begin{align} \notag
	\begin{split}
	&\bdy_sw_{1}+\kap\Lam^\gam w_{1}+\til{u}_{1}\cdotp\del w_{1}=f-\mu S_m(w_{1}-v),\\
	&\bdy_sw_{2}+\kap\Lam^\gam w_{2}+\til{u}_{2}\cdotp\del w_{2}=f-\mu S_m(w_{2}-v).
	\end{split}
\end{align}
Here $\til{u}_{1}=\Ri^{\perp}w_{1}$, $\til{u}_{2}=\Ri^{\perp}w_{2}$.  We subtract the two equations, denoting $\delta=w_{1}-w_{2}$, to obtain
\begin{align}\label{difference1}
	&\bdy_s\delta+\kap\Lam^\gam \delta+\Ri^{\perp}\delta \cdotp \del\delta+\Ri^{\perp}\delta \cdotp \del w_{2}+\til{u}_{2}\cdotp\del \delta=-\mu S_m\delta. 
\end{align}
Observe that $\psi=-\Lam^{-1}\de\in L^\infty(\R;V_{\s+1})$.  Since \eqref{sqg:da} holds in $L^{\infty}(\mathbb{R}; V_{\s-\gamma})\cap L^{2}_{loc}(\mathbb{R}; V_{\s-{\gamma}/{2}})$, so does \req{difference1}.  In particular, $\psi$ is a valid test function for \req{difference1}.  Thus, upon multiplying \req{difference1} by $\psi$ and integrating over $\T^2$, we obtain
\begin{align}
	\frac{1}{2}\frac{d}{ds}\|\psi\|^{2}_{H^{1/2}}+\kap\|\psi\|^{2}_{H^{\frac{\gamma+1}{2}}}+\mu \|\psi\|^{2}_{H^{1/2}} &=\int_{\T^2} (\til{u}_{2}\cdotp\del \delta)\psi\ {{{}}}+ \mu \int_{\T^2} (T_{m}\delta)\psi\ {{{}}},\notag\\
		&=I+II\label{energy1}
\end{align}
where $T_{m}=I-S_{m}$, and where, by application of the Plancherel theorem and the fact that $\Ri^\perp\de$ is divergence-free, we made use of the facts that
\begin{align}\label{cancel}
	\begin{split}
	\int_{\T^2} (\Ri^{\perp}\delta \cdotp \del\delta) \psi\ {{{}}}&= -\int_{\T^2} ((\Ri^{\perp}\delta)\delta)\cdotp\Ri\delta\ {{{}}}= 0,\\
	\int_{\T^2} (\Ri^{\perp}\delta \cdotp \del w_{2}) \psi\ {{{}}}&= -\int_{\T^2} ((\Ri^{\perp}\delta)w_{2})\cdotp\Ri\delta\ {{{}}}= 0. 
	\end{split}
\end{align}
For $I$, observe that upon integrating by parts, we obtain
	\begin{align}
		I= - \int_{\T^2} (\til{u}_{2}\Lam\psi)\cdotp\del \psi\ {{{}}}\notag.
	\end{align}
Let $1/p+2/q=1$.  Then by using H\"older's inequality, the relation $\til{u}_2=\Ri^\perp w_2$, the Calder\'on-Zygmund theorem, the Sobolev embedding theorem for $H^{\frac{1}{p}}\imb L^{q}$, Proposition \ref{prop:goodlp}, and interpolation and Young's inequalities, we may estimate $I$ as
\begin{align*}
	I&\leq C\|\til{u}_{2}\|_{L^p} \|\Lam\psi\|_{L^q} \|\del \psi\|_{L^q}
							   \leq C\|w_{2}\|_{L^{p}}\|\Lam\psi\|_{L^q} \|\del \psi\|_{L^q}
							    \leq CG_{L^p} \|\psi\|^{2}_{H^{1+\frac{1}{p}}} \\ 
							    &\leq CG_{L^p} \|\psi\|^{\frac{2(1+\frac{2}{p})}{\gamma}}_{H^{\frac{\gamma+1}{2}}} \|\psi\|^{\frac{2(\gamma-1-\frac{2}{p})}{\gamma}}_{H^{1/2}} 
							    \leq \frac{\kap}{4}\|\psi\|^{2}_{H^{\frac{\gamma+1}{2}}}+ C\kap\left(\frac{G_{L^p}}{\kap}\right)^{\gam/(\gam-1-2/p)} \|\psi\|^2_{H^{1/2}}, 
\end{align*}

We use the Bernstein inequalities to estimate $II$ as
\begin{align*}
	II&\leq \mu \|T_{m}\delta\|_{H^{-\frac{\gamma+1}{2}}}\|\psi\|_{H^{\frac{\gamma+1}{2}}}\leq C\mu \|T_{m}\psi\|_{H^{\frac{-\gamma+1}{2}}}\|\psi\|_{H^{\frac{\gamma+1}{2}}}\\
					    &\leq C\frac{\mu}{2^{m\frac{\gamma}{2}}}\|\psi\|_{H^{1/2}}\|\psi\|_{H^{\frac{\gamma+1}{2}}}\leq C^2\frac{\mu^2}{2^{m\gam}\kap}\|\psi\|_{H^{1/2}}^{2}+\frac{\kap}{4}\|\psi\|_{H^{\frac{\gamma+1}{2}}}^{2}.
\end{align*}
Now we combine above estimates with \req{energy1} and condition \req{MainCondition} which is the standing hypotheses $(H6)$ to obtain
\begin{align*}
	\frac{1}{2}\frac{d}{ds}\|\psi\|^{2}_{H^{1/2}}+\frac{\mu}{2}\|\psi\|^{2}_{H^{1/2}} \leq 0. 
\end{align*}

Let $s_{0}<s$ be any real number.  By Gronwall's inequality over $[s_{0}, s]$, we get
\begin{align*}
	\|\psi(s)\|^{2}_{H^{1/2}} \leq \|\psi(s_{0})\|^{2}_{H^{1/2}} e^{-\mu(s-s_{0})}. 
\end{align*}
Since $\sup_{s_0\in\R}\|\psi(s_{0})\|^{2}_{H^{1/2}}<\infty$, we take $s_{0}\goesto -\infty$ to obtain that 
\begin{align*}
	\|\psi(s)\|^{2}_{H^{1/2}} =0. 
\end{align*}
Since $\psi$ has mean zero, this implies $\delta(x,s)=-\Lam\psi(x,s)=0$ for a.e. $x\in\T^2$ and for all $s\in \mathbb{R}$.   Hence, $w_1=w_2$ in $Y$.
\end{proof}

\begin{rmk}
We should emphasize that the crucial step in the proof of Proposition \ref{Uniqueness} (and of Theorem \ref{W-Lip-2} to follow) is the application of Proposition \ref{prop:goodlp}.  Indeed, the bulk of the analysis in this paper is devoted to the proof of Proposition \ref{prop:goodlp}.
\end{rmk}

\subsection{Lipschitz property}
Finally, to apply Theorem \ref{thm:gen:detform} and guarantee that \req{df} defines an ODE with locally Lipschitz vector field, we must show that $S_mW:B_X^{4R}(0)\goesto Y$ is a Lipschitz map, where $R$ is given by \req{R:rad}.  In fact, we show more. We show that $W$ itself is a Lipschitz map in an appropriate topology (see \req{W:lip1} below). Then by the boundedness of $S_{m}$, we get that $S_mW$ is a Lipschitz map.

\begin{prop}[$W$ Lipschitz]\label{W-Lip-2}
Assume standing hypotheses $(H1)-(H7)$ hold.  There exists an absolute constant $c_1>0$ such that 
$W: B_{X}^{4R}(0)\goesto C_{b}(\mathbb{R};V_{-{1}/{2}})$ is a Lipschitz function with global Lipschitz constant $c_{1}$. Moreover, $S_{m}W: B_{X}^{4R}(0)\goesto Y$ is a Lipschitz function with global Lipschitz constant
\begin{align}\label{LW}
	L_{W}=c_12^{m(\sigma+\frac{1}{2})}.
\end{align}
\end{prop}

\begin{proof}
Suppose $w_1$ and $w_2$ are the solutions of \req{sqg:da} in $Y$ corresponding to $v_{1}$ and $v_{2}$ belonging both to $B_{X}^{4R}(0)$:
\begin{align} \notag
	&\bdy_sw_{1}+\kap\Lam^\gam w_{1}+\til{u}_{1}\cdotp\del w_{1}=f-\mu S_m(w_{1}-v_{1}),\\
	&\bdy_sw_{2}+\kap\Lam^\gam w_{2}+\til{u}_{2}\cdotp\del w_{2}=f-\mu S_m(w_{2}-v_{2}).\notag
\end{align}
Here $\til{u}_{1}=\Ri^{\perp}w_{1}$, $\til{u}_{2}=\Ri^{\perp}w_{2}$ and $s\in \mathbb{R}$. Subtract, denoting $\delta=w_{1}-w_{2}$ and $\z=v_{1}-v_{2}$, to obtain
\begin{align}\label{difference3}
	&\bdy_s\delta+\kap\Lam^\gam \delta+\Ri^{\perp}\delta \cdotp \del\delta+\Ri^{\perp}\delta \cdotp \del w_{2}+\til{u}_{2}\cdotp\del \delta=-\mu S_m(\delta- \z). 
\end{align}
We multiply \req{difference3} by $\psi=-\Lam^{-1}\delta$, which, as before, is an appropriate test function, and integrate over $\T^2$ to obtain
\begin{align}\label{energy3}
	\frac{1}{2}\frac{d}{ds}\|\psi\|^{2}_{H^{1/2}}+\kap\|\psi\|^{2}_{H^{\frac{\gamma+1}{2}}}+\mu \|\psi\|^{2}_{H^{1/2}} &=\int_{\T^2} (\til{u}_{2}\cdotp\del \delta)\psi\ {{{}}}+ \mu\int_{\T^2} T_{m}\delta \psi\ {{{}}}+ \mu\int_{\T^2} \z \psi\ {{{}}}, \notag \\
									      &=I+II+III,
\end{align}
where $T_{m}=I-S_{m}$, and where we used \req{cancel}.

We estimate $I$ as before.
\begin{align*}
	|I|\leq \frac{\kap}{4}\|\psi\|^{2}_{H^{\frac{\gamma+1}{2}}}+ C\kap\left(\frac{G_{L^p}}{\kap}\right)^{\gam/(\gam-1-2/p)}\|\psi\|^2_{H^{1/2}}
\end{align*}
Also, 
\begin{align*}
	|II|\leq \frac{C\mu^2}{2^{m\gam}\kap}\|\psi\|^{2}_{H^{1/2}}+ \frac{\kap}4 \|\psi\|^2_{H^{\frac{\gamma+1}{2}}}, 
\end{align*}
and
\begin{align*}
	|III|&\leq \mu\|\z\|_{H^{-1/2}}\|\psi\|_{H^{1/2}}\leq \mu\|\z\|_{H^{-1/2}}^{2}+ \frac{\mu}{4}\|\psi\|_{H^{1/2}}^{2}. 
\end{align*}

We return to \req{energy3} and combine $I- III$ with \req{MainCondition} which is the standing hypotheses $(H6)$ and Sobolev embedding  estimate

\begin{align*}
	\frac{1}{2}\frac{d}{ds}\|\psi\|^{2}_{H^{1/2}}+\frac{\kap}2\|\psi\|^{2}_{H^{\frac{\gamma+1}{2}}}+\frac{\mu}2\|\psi\|^{2}_{H^{1/2}}&\leq \mu \|\z\|_{H^{-1/2}}^{2}\leq C\mu\|\z\|_{H^{\sigma}}^{2}\leq C\mu\Sob{\z}{X}^{2}.	
\end{align*}
We apply Gronwall's inequality between $[s_{0}, s]$, and observe that $\Sob{\psi(s_{0})}{H^{1/2}}$ is bounded for any $s_{0}\in \mathbb{R}$.  Thus, by taking $s_{0}\rightarrow -\infty$, we obtain
\begin{align}\label{W:lip1}
	\Sob{W(v_1)-W(v_2)}{H^{-1/2}}^2=\Sob{\Lam^{-1}\de}{H^{1/2}}^2=\|\psi(s)\|^{2}_{H^{1/2}}\leq C\Sob{\z}{X}^{2},
\end{align}
for some absolute constant $C>0$.  In particular, we have
\begin{align*}
	\|S_{m}W(v_{1})(s)-S_{m}W(v_{2})(s)\|_{H^{\sigma}}&=\|S_{m}w_{1}(s)-S_{m}w_{2}(s)\|_{H^{\sigma}} \notag\\
										&\leq C\|S_{m}\psi(s)\|_{H^{\s+1}}\\
									       &\leq C2^{m(\s+\frac{1}{2})}\|S_{m}\psi(s)\|_{H^{1/2}}\\
									       &\leq C2^{m(\s+\frac{1}{2})}\|\psi(s)\|_{H^{1/2}}\\
									       &\leq C2^{m(\s+\frac{1}{2})}\Sob{\z}{X}=C2^{m(\s+\frac{1}{2})}\Sob{v_{1}-v_{2}}{X}.			
\end{align*}
Thus, 
\begin{align*}
	\Sob{S_{m}W(v_{1})-S_{m}W(v_{2})}{Y}\leq {L}_{W}\Sob{v_{1}-v_{2}}{X}, 			
\end{align*}
where ${L}_{W}=C2^{m(\s+\frac{1}{2})}$, as desired.
\end{proof}

\section{A priori estimates: $\eps$-independent bounds}\label{sect:apriori}

As in section \ref{sect:detform}, \textit{we operate under the Standing Hypotheses}, $(H1)-(H7)$.  For clarity, we will indicate the origin of the conditions stated there on $\mu, m$ by emphasizing them in the propositions in which they are needed.  


\subsection{$L^2$, $m$-dependent $L^p$, and $H^\s$ uniform bounds} 

For each $\eps>0$, let $w^\eps\in C(\R; V_{3/2})\cap L^\infty(\R; V_2)$ be the unique strong solution to
	\begin{align}\label{nd:sqg:fbPar}
		\begin{split}
		&\bdy_sw^\eps-\eps\De w^\eps+\kap\Lam^\gam w^\eps+\til{u}^\eps\cdotp\del w^\eps=f^\eps-\mu S_m(w^\eps-v),\quad\til{u}^\eps=\Ri^\perp w^\eps,\quad v\in B^{4R}_X(0),
		\end{split}
	\end{align}
where $R>0$ is given by \req{R:rad}. 
We state the following three propositions, which furnish $\eps$-independent, but $m$-\textit{dependent} bounds for the family $\{w^\eps\}_{\eps>0}$.  These bounds were invoked in Step 3 of the proof of Proposition \ref{da:exist} to ensure that the limiting function, $w\in Y$, of a subsequence of $(w^\eps)_{\eps>0}$ exists.  Their proofs follow exactly as in those of the corresponding propositions in \cite{jmt:sqgda} and the fact that $w^\eps(s_0)e^{-\eps(s-s_0)}\goesto0$ as $s_0\goesto-\infty$, in the respective topology (cf. Propositions 6, 7, and 9 of \cite{jmt:sqgda}).  However, to show that the limiting function $w$ is a unique solution, so that the map $W:v\mapsto w$ is well-defined, we ultimately require estimates that are $m$-\textit{independent}.  We deal with this issue in section \ref{sect:linfty}.  

Let $F_{H^{-\gam/2}}, \Tht_{H^\s}$ given by \req{fgam} and \req{sqg:hs:ball}, respectively.  Define
\begin{align}\label{g2}
		G_{L^2}^2:=C\left(\frac{\kap}{\mu}{F_{H^{-\gam/2}}^2}+\Tht_{H^\s}^2\right).
	\end{align}

\begin{prop}\label{prop:l2:da}
Assume that $(H1)-(H7)$ hold and let $w^\eps$ be the solution of \req{nd:sqg:fbPar}.  There exist absolute constants $C_0, c_0>0$ with $c_0$ depending on $C_0$ such that if \req{g2} holds with $C=C_0$ and the first part of the hypotheses $(H7)$ is satisfied, namely, $\mu$ and $m$ satisfy
	\begin{align}\label{l2:modes}
		2^{m\gam}\geq c_0\frac{\mu}{\kap},
	\end{align}
then
	\begin{align}\label{l2:bounds}
		\sup_{s\in\R}\Sob{w^\eps(s)}{L^2}
		\leq {G_{L^2}}.
	\end{align}
Moreover, the following energy inequality holds:
	\begin{align}\label{l2:ineq}
		\Sob{w^\eps(s_2)}{L^2}^2+\kap\int_{s_1}^{s_2}\Sob{w^\eps(\tau)}{H^{\gam/2}}^2\ d\tau\leq \Sob{w^\eps(s_1)}{L^2}^2+\mu(s_2-s_1)G_{L^2}^2,\quad -\infty<s_1< s_2<\infty.
	\end{align}
\end{prop}

Let $F_{L^p}, G_{L^2}$ be given as in \req{fp} and \req{g2}, respectively.  Define 
	\begin{align}\label{Rp}
		\til{G}_{L^p}^p:=C\left(\frac{\kap^p}{\mu^p}\left(F_{L^p}^p+\left(\frac{G_{L^2}}{p}\right)^p\right)+\Tht_{L^p}^p+C(m,p)^pG_{L^2}^p\right),
	\end{align}
where $C>0$ is an absolute constant (to be specified) and
	\begin{align}\notag
		C(m,p)^p:=2^{m(p-2)}.
	\end{align}

\begin{prop}\label{prop:lp}
Assume that $(H1)-(H7)$ hold and let $w^\eps$ be the solution of \req{nd:sqg:fbPar}.  There exist absolute constants $C_0, c_0$, independent of $m$, with $c_0$ depending on $C_0$, such that if \req{Rp} holds with $C=C_0$, then for $\s, p$ given by $(H3)$, we have
	\begin{align}\label{mp0:final}
		\sup_{s\in\R}\Sob{{w^\eps}(s)}{L^p}\leq p\frac{\mu}{\kap}\til{G}_{L^p}.
	\end{align}
\end{prop}

Finally, let $G_{L^2}$ be given as in \req{g2} and $\Tht_{H^\s}$ by \req{sqg:hs:ball}.  Define
	\begin{align}\label{R:sig}
		\begin{split}
		{{F}_{H^{\s-\gam/2}}}:=\frac{1}{{\kap}}\Sob{f}{H^{\s-\gam/2}},\quad \Xi_{r,\al}:=C\left(\frac{{\sup_{s\in\R}}\Sob{w^\eps(s)}{L^r}}{\kap}\right)^{\frac{2\al}{\gam-1-2/r}},\quad
		\end{split}
	\end{align}
and
	\begin{align}\label{g:sig:til}
		{\til{G}_{H^\s}}^2:=C\left(\frac{\kap}{\mu}{{F}_{H^{\s-\gam/2}}^2}+\Tht_{H^\s}^2+\Xi_{p,\s}G_{L^2}^2\right).
	\end{align}

\begin{prop}\label{prop:hs:a priori}
Assume that $(H1)-(H7)$ hold and let $w^\eps$ be the solution of \req{nd:sqg:fbPar}.  There exists absolute constants $C_0, c_0>0$, with $c_0$ depending on $C_0$ such that if \req{R:sig}, \req{g:sig:til} hold with $C=C_0$, and the first part of the standing hypotheses $(H7)$ is satisfied, namely, $\mu$ and $m$ satisfy
	\begin{align}\label{hs:modes}
		2^{\gam m}\geq c_0\frac{{{\mu}}}{{{\kap}}},
	\end{align}
then
	\begin{align}\notag
		\sup_{s\in\R}\Sob{w^\eps(s)}{H^\s}\leq {\til{G}_{H^\s}},
	\end{align}
and
	\begin{align}\notag
		\kap\int_{s_1}^{s_2}\Sob{w^\eps(\tau)}{H^{\s+\gam/2}}^2\ d\tau\leq\Sob{w^{\eps}(s_1)}{H^\s}^2+\Xi_{p,\s}G_{L^2}^2+\mu(s_2-s_1)\til{G}_{H^\s}^2,\quad -\infty<s_1<s_2<\infty.
	\end{align}
\end{prop}

\begin{rmk}
Note that we denote $\til{G}_{L^p}, {\til{G}_{H^\s}}$, decorated with $\sim$, to emphasize the potential dependence on $m$.  Indeed, if one applies the bounds of Proposition \ref{prop:lp} to those in Proposition \ref{prop:hs:a priori}, then ${\til{G}_{H^\s}}$ will also depend on $m$. 
\end{rmk}

As we mentioned earlier, although the bounds of Proposition \ref{prop:lp} and \ref{prop:hs:a priori} are $\eps$-independent, they are \textit{still insufficient} to show that the map $W: v\mapsto w$ is Lipschitz as a map from $B^{4R}_X(0)\goesto H^{-1/2}$, or even well-defined as a map from $B^{4R}_X(0)\goesto Y$.  Indeed, without improved bounds, we would instead have $\til{G}_{L^p}=\til{G}_{L^p}(m,\mu)$ in place of $G_{L^p}$, which \textit{is} independent of $m$ and $\mu$ (see \req{Gp}); this would make it impossible to simultaneously satisfy \req{MainCondition}, which is the standing hypotheses $(H6)$, and \req{modes:crucial}, which is the standing hypotheses $(H7)$. 
For this reason, we will furthermore show that $\eps$-independent $L^\infty$-bounds are available for the family $\{w^\eps\}$. 
In section \ref{sect:refine}, we finally show that these $L^\infty$ bounds can be used to obtain a proper refinement of Proposition \ref{prop:lp} and \ref{prop:hs:a priori} to furnish bounds which are independent of $\eps, m$, \textit{and} $\mu$.



\subsection{$L^\infty$ estimates}\label{sect:linfty}
In this section, we obtain the desired $L^\infty$-bounds for $w^\eps$ by De Giorgi iteration.  Our estimates will follow along the lines of \cite{chesk:dai:critqg}.  We emphasize again that we will assume that the Standing Hypotheses $(H1)-(H7)$ hold throughout and that $w^\eps$ is the solution of \req{nd:sqg:fbPar}.

We will obtain estimates for $\sup_{s\in\R}\Sob{w^\eps(s)}{L^\infty}$ that are \textit{independent of}\ $\eps$ and $m$.  In particular, our main claim in this section is the following.

\begin{prop}\label{prop:dg}
Assume that $(H1)-(H7)$ hold and let $w^\eps$ be the solution of \req{nd:sqg:fbPar}.  Let $\rho_0>0$.  There exist absolute constants $C_0,c_0>0$, with $c_0$ depending on $C_0$, such that if 
	\begin{align}\label{v:bound}
		\sup_{s\in\R}\Sob{v(s)}{L^p}\leq \rho_0,
	\end{align}
for some $1-\s<2/p<\gam-1$, which is the standing hypotheses $(H3)$, and the first part of the standing hypotheses $(H7)$ is satisfied, namely, $\mu$ and $m$ satisfy
	\begin{align}\notag
		2^{m\gam}\geq c_0\frac{\mu}{\kap},
	\end{align}
then for any $\de_\infty>0$, we have
	\begin{align}\label{dg:bound}
		\sup_{s\in\R}\Sob{w^\eps(s)}{L^\infty}\leq C\left ( \max \left \{	\frac{1}{\de_\infty \kap},	\frac{\mu}{\kap}\left(\frac{\kap}{\mu}F_{L^p}+ \rho_0 + 1\right)	\right \}\left (U_{0}^{\frac{\gamma}{2}}+U_{0}^{\frac{\gam}{2}-\frac{1}{p}}\right )\right )^{\frac{1}{\gam- \frac{2}{p}}},
	\end{align}
where $U_0\leq 4\mu\de_\infty G_{L^2}^2$.
\end{prop}
Observe that by $(H3)$, the Sobolev embedding theorem, and $(H5)$ we have
	\begin{align}\notag
		\sup_{s\in\R}\Sob{v(s)}{L^p}\leq 4R.
	\end{align}
Thus, by letting $\rho_0=4R$, $\de_\infty=1/\mu$ in Proposition \ref{prop:dg}, and 
	\begin{align}\label{Gp}
		G_{L^p}:=\frac{\kap}{\mu}F_{L^p}+\Tht_{H^\s}\quad\text{and}\quad M_\infty:=  C\left(\frac{\mu}{\kap}\right)^{\frac{1}{\gam-\frac{2}p}}\left[\left(\frac{\kap}{\mu}F_{L^p}+ \rho_0+1\right)\left (U_{0}^{\frac{\gamma}{2}}+U_{0}^{\frac{\gam}{2}-\frac{1}{p}}\right )\right]^{\frac{1}{\gam- \frac{2}{p}}},
	\end{align}
where $U_0\leq 4\mu\de_\infty G_{L^2}^2$, we immediately obtain the following:

\begin{coro}\label{coro:dg}
Assume that $(H1)-(H7)$ hold and let $w^\eps$ be the solution of \req{nd:sqg:fbPar}.  If the first part of the standing hypotheses $(H7)$ is satisfied, namely, $\mu$ and $m$ satisfy
	\begin{align}\notag
		2^{m\gam}\geq c_0\frac{\mu}{\kap},
	\end{align}
then
	\begin{align}\label{m:infty}
		\sup_{s\in\R}\Sob{w^\eps(s)}{L^\infty}\leq M_\infty,
	\end{align}
where $M_\infty$ is given as in \req{Gp}.
\end{coro}

Before proving Proposition \ref{prop:dg}, we make the following reduction.  Observe that it suffices to consider restrictions of $w^\eps$, i.e., 
	\begin{align}\label{wepsk}
		w^\eps_k=w^\eps|_{[-k,\infty)},\quad k>0,
	\end{align}
and to establish \req{dg:bound} with constants \textit{independent of} $k$.  We will then derive a level-set energy inequality, which we will exploit in the proof of Proposition \ref{prop:dg} to obtain $L^\infty$ estimates on $w^\eps_k(s)$, which are uniform in $s\in\R$ and $\eps, k>0$.  

To this end, given $\lam>0$, we define
	\begin{align}\label{lvl:set}
		\ph(\xi):=(\xi-\lam)\ind_{\{\xi>\lam\}}=\max\{\xi-\lam,0\}=(\xi-\lam)_+.
	\end{align}
Define also the vector function $\Ph(\xi)$ by
	\begin{align}\label{lvl:set:vf}
		\Ph(\xi):=(\ph(\xi), \ph(-\xi)).
	\end{align}
Then we will prove the following using De Giorgi techniques as in \cite{chesk:dai:critqg}:

\begin{prop}\label{prop:lvlsets}
Assume that $(H1)-(H7)$ hold and let $w^\eps$ be the solution of \req{nd:sqg:fbPar}.  Let $\lam>0$ and  $\ph, \Ph$ be given by \req{lvl:set}, \req{lvl:set:vf}, respectively.   Let $k>0$ and $s_0\geq-k$.  There exists an absolute constant $c_0>0$, independent of $\lam, k, s_0$, such that if the first part of the standing hypotheses $(H7)$ is satisfied, namely, $\mu$ and $m$ satisfy
	\begin{align}\label{modes:lvlsets}
		2^{\gam m}\geq c_0\frac{\mu}{\kap},
	\end{align}
then
	\begin{align}\label{lvlset:ineq}
		\begin{split}
		\Sob{\Ph(w^\eps_k)(s_2)}{L^2}^2+&\kap\int_{s_1}^{s_2}\Sob{\Ph(w^\eps_k)(\tau)}{\dot{H}^{\gam/2}}^2\ d\tau\\
			\leq &\Sob{\Ph(w^\eps_k)(s_1)}{L^2}^2+2\sqrt{2}\int_{s_1}^{s_2}\int_{\T^2}\left|f+\mu S_m v(\tau)\right||\Ph(w^\eps_k)(\tau)|\ {{{}}}\ d\tau\\
			&+ 2\mu\int_{s_1}^{s_2}\int_{\T^2} {\Sq} (w^\eps_k)_\lam\til{\Sq}\ph(w^\eps_k)\ {{{}}}\ d\tau+2\mu\int_{s_1}^{s_2}\int_{\T^2} {\Sq}(- w^\eps_k)_\lam\til{\Sq}\ph(- w^\eps_k)\ {{{}}}\ d\tau,
		\end{split}
	\end{align}
holds for all $s_0<s_1\leq s_2<\infty$, where $\Sq, \til{\Sq}$ are given by \req{lp:sq}.
\end{prop}

To prove this, we will make use of the following elementary decomposition.

\begin{lemmm}\label{lem:decomp}
Let $\lam>0$ and $\ph$ be given by \req{lvl:set}.  Let $g$ be a Lebesgue measurable function over $\T^2$.  Define
	\begin{align}\label{g:def}
		g_\lam:=g\ind_{\{|g|\leq\lam\}}+\lam\left(\ind_{\{g>\lam\}}-\ind_{\{-g>\lam\}}\right).
	\end{align}
Then
	\begin{align}\label{g:bound}
		|g_\lam|\leq\lam,
	\end{align}
and
	\begin{align}\label{g:decomp}
		 g=\ph(g)-\ph(- g)+ g_\lam,
	\end{align}
where $\ph$ is defined by \req{lvl:set}.
\end{lemmm}

\begin{proof}
The bound \req{g:bound} follows from the fact that $\{|g|\leq\lam\}, \{g>\lam\}, \{-g>\lam\}$ forms a partition of $\T^2$.  On the other hand, by definition of \req{lvl:set} and \req{g:def}, observe that
	\begin{align}
		g&= g\ind_{\{|g|\leq\lam\}}+ g\ind_{\{|g|>\lam\}}= g\ind_{\{|g|\leq\lam\}}+ g\ind_{\{g>\lam\}}+ g\ind_{\{- g>\lam\}}=\ph( g)-\ph(- g)+ g_\lam.\notag
	\end{align}
\end{proof}


\begin{proof}[{Proof of Proposition \ref{prop:lvlsets}}]   For convenience, we will simply denote $w^\eps_k$ by $w$.
Before performing the estimates, we make the following three observations.

Firstly, since $\ph'(\xi)=\ind_{[\xi>\lam]}$ a.e. $\xi$, it follows that
	\begin{align}\label{ph:id}
		\ph'(\xi)\ph(\xi)=\ph(\xi)\quad \text{a.e.}\ \xi.
	\end{align}
Secondly, since $w$ is periodic, $\ph( w)$ is also periodic, so that
	\begin{align}\label{laplace:pos:+dg}
		-\int_{\T^2}\De w\ph( w)\ph'( w)\ {{{}}}=-\int_{\T^2}\De w\ph( w)\ {{{}}}=\int_{\T^2}|\del w|^2\ph'( w)\ {{{}}}\geq0.
	\end{align}
Thirdly, observe that
	\begin{align}\label{decomp1}
		-S_m w=T_m w- w=T_m w-( w-\lam)-\lam.
	\end{align}

We now proceed with the following energy estimates.  Multiplying \req{nd:sqg:fbPar} by $\ph'( w)\ph( w)$, integrating in $x$, and applying \req{ph:id}-\req{decomp1}, we obtain
	\begin{align}\label{dg:balance1}
		\frac{1}{2}&\frac{d}{ds}\int_{\T^2} \ph^2( w)\ {{{}}}+\kap\int_{\T^2}(\Lam^\gam w)\ph'( w)\ph( w)\ {{{}}}+\mu\int_{\T^2}\ph^2(w)\ {{{}}}+\mu\lam\int_{\T^2}\ph(w)\ {{{}}}\notag\\
				&=-\int_{\T^2} ((\til{u}\cdotp\del)w)\ph'( w)\ph( w)\ {{{}}}+\int_{\T^2} (f+\mu S_m v)\ph( w)\ {{{}}}+\mu\int_{\T^2} T_mw\ph(\w)\ {{{}}}\notag\\
				&=I+II+III
	\end{align}
%
We claim that  $I=0$.  Indeed, since $\del\cdotp \til{u}=0$, we have
	\begin{align}\notag
		\frac{1}2\bdy_j(\til{u_{j}}\ph^2( w))=\frac{1}2(\bdy_j\til{u}_j)\ph^2( w)+\til{u}_j\ph( w)\bdy_j(\ph( w))=\til{u}_j\ph( w)\ph'( w)\bdy_j w,
	\end{align}
where we have summed over repeated indices.  It then follows that
	\begin{align}\label{I:lvlset}
		(\til{u}\cdotp\del w)\ph'( w)\ph( w)=	\frac{1}2\del\cdotp(\til{u}\ph^2( w)),
	\end{align}
which implies that $I=0$. For the term $III$, we apply Lemma \ref{lem:decomp}, so that
	\begin{align}\label{w:decomp}
		|w_\lam|\leq\lam\quad\text{and}\quad w=\ph(w)-\ph(-w)+w_\lam.
	\end{align}
Hence, by Cauchy-Schwarz we have
	\begin{align}\label{III:lvlset}
		III&=\mu\sum_{j>m}\int_{\T^2} (\lpj w_\lam)(\lpjc\ph( w))\ {{{}}}+\mu\int_{\T^2} \left(T_m\ph( w)-T_m\ph(- w)\right)\ph( w)\ {{{}}}\notag\\
				&\leq\mu\int_{\T^2} ({\Sq} w_\lam)(\til{\Sq}\ph( w))\ {{{}}}+\mu\int_{\T^2} \left(T_m\ph( w)-T_m\ph(- w)\right)\ph( w)\ {{{}}}.
	\end{align}
We treat the dissipation term in \req{dg:balance1} by invoking \req{ph:id} and Proposition \ref{lem:lvlsetpos}, so that
	\begin{align}\label{pos:lem:dg}
		\int_{\T^2}(\Lam^\gam w)\ph( w)\ph'( w)=\int_{\T^2}(\Lam^\gam w)\ph( w)\geq \int_{\T^2}(\Lam^\gam \ph(w))\ph( w)=\int_{\T^2}\left|\Lam^{\gam/2}\ph( w)\right|^2\ {{{}}}=\Sob{\ph( w)}{\dot{H}^{\gam/2}}^2.
	\end{align}
Combining \req{I:lvlset}, \req{III:lvlset}, and \req{pos:lem:dg} we arrive at
	\begin{align}\label{dg:balance3}
		\frac{1}{2}\frac{d}{ds}\Sob{\ph( w)}{L^2}^2&+\kap\Sob{\ph( w)}{\dot{H}^{\gam/2}}^2+\mu\Sob{\ph( w)}{L^2}^2+\lam\mu\Sob{\ph( w)}{L^1}\notag\\
			&\leq II+\mu\int_{\T^2} {\Sq} w_\lam\til{\Sq}\ph( w)\ {{{}}}+\mu\int_{\T^2}\left(T_m\ph( w)-T_m\ph(- w)\right)\ph( w)\ {{{}}}
	\end{align}
We now apply the same argument in deriving \req{dg:balance3} with $- w$, replacing $w$.  Indeed, observe that $- w$ satisfies
	\begin{align}\notag
		\bdy_s(- w)-\eps\De(- w)+\kap\Lam^{\gam}(- w)-(-\til{u})\cdotp\del(- w)=-f-\mu S_m(- w-(- v)).
	\end{align}
Since
	\begin{align}\label{laplace:pos:-dg}
		-\int_{\T^2}(\De(- w))\ph(- w)\ph'(- w)\ {{{}}}=-\int_{\T^2}(\De(- w))\ph(- w)\ {{{}}}=\int_{\T^2}|\del(- w)|^2\ph'(- w)\ {{{}}}\geq0,
	\end{align}
also holds, we may argue as before to arrive at
	\begin{align}\label{dg:balance4}
		\frac{1}{2}&\frac{d}{dt}\Sob{\ph(- w)}{L^2}^2+\mu\Sob{\ph(- w)}{L^2}^2+\kap\Sob{\ph(- w)}{\dot{H}^{\gam/2}}^2+\lam\mu\Sob{\ph(- w)}{L^1}\notag\\
		&\leq-\int_{\T^2}\left(f+\mu S_m v\right)\ph(- w)\ {{{}}}\notag\\
					&+\mu\int_{\T^2} {\Sq}(- w)_\lam\til{\Sq}\ph(- w)\ {{{}}}-\mu\int_{\T^2}\left(T_m\ph( w)-T_m\ph(- w)\right)\ph(- w)\ {{{}}}.
	\end{align}
Adding \req{dg:balance3} and \req{dg:balance4}, and using \req{lvl:set:vf} we obtain
	\begin{align}\label{total:balance1}
	\frac{1}{2}\frac{d}{dt}\Sob{\Ph(w)}{L^2}^2+\kap\Sob{\Ph(w)}{H^{\gam/2}}^2+\mu\Sob{\Ph(w)}{L^2}^2\leq&
			\mu\int_{\T^2}\left(T_m\ph( w)-T_m\ph(- w)\right)(\ph( w)-\ph(- w))\ {{{}}}\notag\\
			&+\mu\int_{\T^2} ({\Sq} w_\lam)(\til{\Sq}\ph( w))+\mu\int_{\T^2}({\Sq}(- w)_\lam)(\til{\Sq}\ph(- w))\ {{{}}}\notag\\
			&+\int_{\T^2}\left(f+\mu S_m v\right)(\ph( w)-\ph(- w))\ {{{}}},
	\end{align}
where $\Sq, \til{\Sq}$ are given by \req{lp:sq}. We focus on the first term on the right-hand side of \req{total:balance1}.  Observe that we may estimate as we did for $II$ in the proof of Proposition \ref{Uniqueness} to obtain
	\begin{align}
		\mu\int_{\T^2} T_m\ph( w)\ph( w)\ {{{}}}\leq \frac{\kap}{16}\Sob{\ph( w)}{\dot{H}^{\gam/2}}^2+\frac{C}{2^{m\gam}}\frac{\mu^2}{\kap}\Sob{\ph( w)}{L^2}^2.\notag
	\end{align}
Similarly
	\begin{align}
		&\mu\int_{\T^2} T_m\ph(- w)\ph( w)\ {{{}}}\leq \frac{\kap}{16}\Sob{\ph(- w)}{\dot{H}^{\gam/2}}^2+\frac{C}{2^{m\gam}}\frac{\mu^2}{\kap}\Sob{\ph( w)}{L^2}^2,\notag\\
		&\mu\int_{\T^2} T_m\ph( w)\ph(- w)\ {{{}}}\leq \frac{\kap}{16}\Sob{\ph( w)}{\dot{H}^{\gam/2}}^2+\frac{C}{2^{m\gam}}\frac{\mu^2}{\kap}\Sob{\ph(- w)}{L^2}^2,\notag\\
		&\mu\int_{\T^2} T_m\ph(- w)\ph(- w)\ {{{}}}\leq \frac{\kap}{16}\Sob{\ph(- w)}{\dot{H}^{\gam/2}}^2+\frac{C}{2^{m\gam}}\frac{\mu^2}{\kap}\Sob{\ph(- w)}{L^2}^2.\notag
	\end{align}
We may therefore absorb these four terms into the left-hand side of \req{total:balance1} provided that \req{modes:lvlsets}, i.e., the first part of the standing hypothesis $(H7)$, holds.  Hence, we arrive at 
	\begin{align}\label{total:balance2}
		\begin{split}
	\frac{d}{ds}\Sob{\Ph(w)}{L^2}^2+\kap\Sob{\Ph(w)}{H^{\gam/2}}^2\leq&2\mu\int_{\T^2} {\Sq} w_\lam\til{\Sq}\ph( w)+2\mu\int_{\T^2} {\Sq}(- w)_\lam\til{\Sq}\ph(- w)\ {{{}}}\\
				&+2\int_{\T^2}\left(f+\mu S_m v\right)(\ph( w)-\ph(- w))\ {{{}}}.
		\end{split}
	\end{align}
Finally, integrating \req{total:balance2} over $[s_1,s_2]$ and applying Cauchy-Schwarz yields \req{lvlset:ineq}, thus completing the proof of Proposition \ref{prop:lvlsets}.
\end{proof}

To prove Proposition \ref{prop:dg}, we will make use of the following three lemmas.  To help make clear the ideas surrounding Proposition \ref{prop:dg}, we will defer the elementary proofs of the first two lemmas to the appendix.  However, we prove the third as it is central to the technique we use.

\begin{lemmm}\label{prop:gen:ineq}
Suppose $\{g_n\}, \{h_n\}, \{\chi_{g_n}\}, \{\chi_{h_n}\}$ are families of non-negative functions over $\R$ and for all $n\geq0$ satisfy: 
	\begin{enumerate}
		\item$g_n\leq g_{n-1}$ and $h_n\leq h_{n-1}$;
		\item $\chi_{g_{n+1}}\leq C_ng_{n}$ and $\chi_{h_{n+1}}\leq C_nh_{n}$, for some absolute constants $C_n>0$;
		\item $g_{n}(\xi)\neq0$ implies $\chi_{g_n}(\xi)\geq1$;
		\item $h_n(\xi)\neq0$ implies $\chi_{h_n}(\xi)\geq1$.
	\end{enumerate}
Let $1<P, Q<\infty$, $\al\in(0,1)$, and $\gam\in(1,2)$ satisfy
	\begin{align}\label{alpha}
		\frac{1}{P+Q}=\frac{1}2(1-\al)+\frac{(2-\gam)}{4}\al.
	\end{align}
Then
	\begin{align}\label{gen:ineq}
		\Sob{(g_n,h_n)}{L^P}^P\leq C_n^{Q}\Sob{(g_{n-1},h_{n-1})}{L^2}^{(1-\al)(P+Q)}\Sob{(g_{n-1},h_{n-1})}{H^{\gam/2}}^{\al(P+Q)}.
	\end{align}
\end{lemmm}

We will derive a particular nonlinear iteration inequality for the quantity $U_n$ as defined by \req{Uk:def}.  This inequality will ensure that $U_n\goesto0$ under certain conditions, which ultimately implies $L^\infty$ bounds.  We note that such an inequality was also used in \cite{caff:vass, chesk:dai1}, but for our purposes we must carefully track of the dependence on certain parameters. 

\begin{lemmm}\label{lem:nonlin:iter}
Let $\{V_n\}_{n\geq0}$ be a sequence of positive numbers.  Suppose there exist $M>0, a>0, b>0, K>0$, and $d_{j}>0$ , for $j=1,2,\dots,K$, such that $d=\min\{d_{1}, d_{2}, ..., d_{K}\}>\frac{3}{2}$ and
	\begin{align*}
		V_{n}\leq C\frac{2^{na}}{M^{b}}\sum_{j=1}^{K} V_{n-1}^{d_{j}},
	\end{align*}
holds for all $n\geq 1$.  Let 
	\begin{align*}
		y_0:= \frac{3a}{2d-3}.
	\end{align*} 
There exists an absolute constant $c_0>0$ such that if $M$ satisfies
	\begin{align*}
		M\geq  C_0\max \left \{\left (\frac{2^{2(a+y_0)}\sum_{j=1}^{K}V_{1}^{d_{j}}}{V_{0}}\right )^{\frac{1}{b}},\left (\frac{\sum_{j=1}^{K}V_{0}^{d_{j}}}{V_{0}}\right )^{\frac{1}{b}} \right \},
	\end{align*} 
then
	\begin{align*}
		V_{n}\leq \frac{V_{0}}{2^{ny_0}}, 
	\end{align*}
holds for all $n\geq 2$.  Moreover, if $V_{1}\leq V_{0}$ then it suffices to choose $M$ to satisfy
	\begin{align*}
		M\geq C_0\left (2^{2(a+y_0)}\sum_{j=1}^{K}V_{0}^{d_{j}-1}\right )^{\frac{1}{b}}.
	\end{align*} 			
\end{lemmm}

Finally, the third lemma provides control of higher level-set truncations in terms of lower ones.

\begin{lemmm}\label{lem:trunc}
Let $M\geq0$.  For each $n\geq0$, define the truncation levels $\lam_n$ by
	\begin{align}\notag
		\lam_n:=M(1-2^{-n}),
	\end{align}
and truncation number $n$ by
	\begin{align}\notag
		\ph_n(\xi):=(\xi-\lam_n)\ind_{\{\xi>\lam_n\}}.
	\end{align}
Then
	\begin{align}\label{dg:facts}
		\ph_{n}(\xi)\leq\ph_{n-1}(\xi)\quad\text{and}\quad \ind_{\{\xi>\lam_n\}}\leq\frac{2^n}{M}\ph_{n-1}(\xi),
	\end{align}
\end{lemmm}
\begin{proof}
Simply observe that $\lam_n-\lam_{n-1}=M2^{-n}$, so that $\lam_n\geq\lam_{n-1}$ and $\{\xi>\lam_n\}=\{\xi-\lam_{n-1}>M2^{-n}\}\sub\{\xi>\lam_{n-1}\}$.  It immediately follows that $\ph_n(\xi)\leq\ph_{n-1}(\xi)$ and
	\begin{align}\notag
		\ind_{\{\xi>\lam_n\}}=\ind_{\{\xi-\lam_{n-1}>M2^{-n}\}}\leq\frac{2^n}{M}(\xi-\lam_{n-1})\ind_{\{\xi-\lam_{n-1}>0\}},
	\end{align}
as desired.
\end{proof}

Finally, we are ready to prove Proposition \ref{prop:dg}. 

\subsubsection*{Proof of Proposition \ref{prop:dg}} Recall that the goal is to show that the solution $w^\eps$ of \req{nd:sqg:fbPar} satisfies $L^\infty$ bounds \textit{independent} of $\eps, m$. 

Let $M\geq1$ and $\de_\infty>0$.  Fix an arbitrary $k\in\R$ and let
	\begin{align}\label{trans}
			s_\infty:=-k+\de_\infty>-k.
	\end{align}
Let $s_0=-k$.  Define 
	\begin{align}\notag
		s_n:=-k+\de_\infty(1-2^{-n})\quad\text{and}\quad I_{n-1}=[s_{n-1},s_{n}],\quad  n\geq 1.
	\end{align}
In light of \req{lvl:set:vf}, we denote by $\Ph_n(\xi)$ the vector field given by
	\begin{align}\notag
		\Ph_n( \xi):=(\ph_n(\xi), \ph_n(-\xi)),
	\end{align}
where $\ph_n$ is given as truncation number $n$ with truncation levels $\lam_n$ given as in Lemma \ref{lem:trunc}.

Let $w=w^\eps|_{[-k,\infty]}$.  We consider the energy level sets, $U_n(S)$, given by 
	\begin{align}\label{Uk:def}
		U_n(S):=&\sup_{s_n\leq s\leq S}\Sob{\Ph_n( w)(s)}{L^2}^2+\kap\int_{s_n}^{S} \Sob{\Ph_n(w)(\tau)}{{H}^{\gam/2}}^2\ d\tau,
	\end{align}
and set $S=s_\infty$.  Observe that if $U_n(s_\infty)\goesto0$, as $n\goesto\infty$, then $\Sob{w(s_\infty)}{L^\infty}\leq M$.  We will show that $U_n(s_\infty)$ satisfies a nonlinear iteration inequality that will imply $U_{n}(s_\infty)$ converges to $0$, as $n\goesto\infty$.  From now on, for convenience, let us simply denote $U_n=U_n(s_\infty)$.



Let $s\in I_{n-1}$ and $s'\in[s_n,s_\infty]$.  Then by Proposition \ref{prop:lvlsets} it follows that
	\begin{align}
		\Sob{\Ph(w)(s')}{L^2}^2+&\kap\int_{s}^{s'}\Sob{\Ph(w)(\tau)}{\dot{H}^{\gam/2}}^2\ d\tau\notag\\
			\leq &\Sob{\Ph(w)(s)}{L^2}^2+2\sqrt{2}\int_{s}^{s'}\int_{\T^2}\left|f+\mu S_m v(\tau)\right||\Ph(w)(\tau)|\ {{{}}}\ d\tau\notag\\
			&+ 2\mu\int_{s}^{s'}\int_{\T^2} {\Sq} (w)_\lam\til{\Sq}\ph(w)\ {{{}}}\ d\tau+2\mu\int_{s_1}^{s_2}\int_{\T^2} {\Sq}(- w)_\lam\til{\Sq}\ph(- w)\ {{{}}}\ d\tau\notag\\
			\leq& \Sob{\Ph_n (w)(s)}{L^2}^2+2\sqrt{2}\int_{s_{n-1}}^{s_\infty}\int_{\T^2}\left|f+\mu S_mv(\tau)\right||\Ph_n( w)(\tau)|\ {{{}}}\ d\tau\notag\\
			&+ 2\mu\int_{s_{n-1}}^{s_\infty}\int_{\T^2} {\Sq}(w_{\lam_n})\til{\Sq} w_n\ {{{}}}\ d\tau+2\mu\int_{s_{n-1}}^{s_\infty}\int_{\T^2} {\Sq}(- w)_{\lam_n}\til{\Sq}(- w)_n\ {{{}}}\ d\tau,\notag
	\end{align}
where $w_{\lam_n}$ is defined as in \req{g:def}.  Therefore, since $[s_{n},s_\infty]\sub[s_{n-1},s_\infty]$, upon taking the supremum over all $s'\in(s_n,s_\infty]$ we obtain
	\begin{align}\label{Uk1}
		U_n\leq & \Sob{\Ph_n (w)(s)}{L^2}^2+2\sqrt{2}\int_{s_{n-1}}^{s_\infty}\int_{\T^2}\left|f+\mu S_mv(\tau)\right||\Ph_n( w)(\tau)|\ {{{}}}\ d\tau\notag\\
			&+ 2\mu\int_{s_{n-1}}^{s_\infty}\int_{\T^2} {\Sq}(w_{\lam_n})\til{\Sq} w_n\ {{{}}}\ d\tau+2\mu\int_{s_{n-1}}^{s_\infty}\int_{\T^2} {\Sq}(- w)_{\lam_n}\til{\Sq}(- w)_n\ {{{}}}\ d\tau.
	\end{align}
Upon taking time averages in $s$ of \req{Uk1} over the interval $I_{n-1}$, we obtain
	\begin{align}\notag
		\begin{split}
		U_n\leq &\frac{2^n}{\de_\infty}\int_{s_{n-1}}^{s_\infty}\Sob{\Ph_n( w)(\tau)}{L^2}^2\ d\tau+2\sqrt{2}\int_{s_{n-1}}^{s_\infty}\int_{\T^2}\left|f+\mu S_mv(\tau)\right||\Ph_n( w)(\tau)|\ {{{}}}\ d\tau\\
			&+ 2\mu\int_{s_{n-1}}^{s_\infty}\int_{\T^2} {\Sq}(w_{\lam_n})\til{\Sq}(\ph_n(w))\ {{{}}}\ d\tau+2\mu\int_{s_{n-1}}^{s_\infty}\int_{\T^2} {\Sq}((- w)_{\lam_n})\til{\Sq}(\ph_n(-w))\ {{{}}}\ d\tau.\\
			=&I+II+III+IV.
		\end{split}
	\end{align}
It will suffice to estimate $I-III$ since $IV$ is similar to $III$.

For $I$, we apply Lemma \ref{prop:gen:ineq} with $g_n=\ph_n(w), h_n=\ph_n(-w), \chi_{g_n}=\ind_{[w>\lam_n]}, \chi_{h_n}=\ind_{[-w>\lam_n]}$, and $C_n=2^n/M$.  This choice is valid by \req{dg:facts}.  Thus, from \req{gen:ineq} and \req{alpha} with $P=2$ and $Q=\gam$, we have $\al=2/(2+\gam)$ and
	\begin{align}\notag
		\begin{split}
		\Sob{\Ph_n(w)}{L^2}^2\leq C\frac{2^{n\gam}}{M^{\gam}}\Sob{\Ph_{n-1}( w)}{L^2}^{\gam}\Sob{\Ph_{n-1}( w)}{{H}^{\gam/2}}^{2},
		\end{split}
	\end{align}
Therefore, upon returning to $I$, we have
	\begin{align}\label{dg:I}
		I=\frac{2^n}{\de_\infty}\int_{s_{n-1}}^{s_\infty}\Sob{\Ph_n( w)(\tau)}{L^2}^2\ d\tau&\leq \frac{2^n}{\de_\infty} \int_{s_{n-1}}^{s_\infty}C\frac{2^{n\gam}}{M^{\gam}}\Sob{\Ph_{n-1}( w)}{L^2}^{\gam}\Sob{\Ph_{n-1}( w)}{{H}^{\gam/2}}^{2}\notag \\
		&\leq \frac{C}{\de_\infty\kap}\frac{2^{n(\gam+1)}}{ M^{\gam}}\left(\sup_{s_{n-1}\leq s\leq s_{\infty}}\Sob{\Ph_{n-1}( w)(s)}{L^2}^{\gam}\right)\int_{s_{n-1}}^{s_\infty}\Sob{\Ph_{n-1}( w)}{{H}^{\gam/2}}^{2}\notag\\
		&\leq\frac{C}{\de_\infty\kap}\frac{2^{n(\gam+1)}}{ M^{\gam}}U_{n-1}^{\gam/2}U_{n-1}.
	\end{align}


For $II$, let $p>2$ and $p'$ the H\"older conjugate of $p$, so that $1<p'<2$.  We recall that $\Sob{v}{L^p}\leq \rho_0$ by \req{v:bound}.  Then by H\"older's inequality, Bernstein's inequality, and Lemma \ref{prop:gen:ineq} with $P=p'$ and $Q=2+p'(\gam-1)$ and \req{alpha}, so that $\alpha=2(1+Q/p')^{-1}$, we have
	\begin{align}\notag
		II&\leq C \left(\Sob{f}{L^p}+\mu\rho_0\right)\int_{s_{n-1}}^{s_\infty} \Sob{\Ph_n( w)(\tau)}{L^{p'}}\ d\tau\notag\\
		&\leq C \left(\Sob{f}{L^p}+\mu\rho_0\right)\frac{2^{nQ/p'}}{M^{Q/p'}}\int_{s_{n-1}}^{s_\infty}\Sob{\Ph_{n-1}(w)(\tau)}{L^2}^{(1-\al)(1+Q/p')}\Sob{\Ph_{n-1}( w)(\tau)}{\dot{H}^{\gam/2}}^{\al(1+Q/p')}\ d\tau\notag\\
		&\leq C\left(F_{L^p}+\frac{\mu}{\kap}\rho_0\right)\frac{2^{n(2/p'+\gam-1)}}{M^{2/p'+\gam-1}}U_{n-1}^{1/p'+\gam/2-1}U_{n-1}.\notag%
	\end{align}

We are left with $III$.  By \req{g:bound}, we have $|w_{\lam_n}|\leq\lam_n\leq M$, so that by H\"older's inequality and Proposition \ref{prop:lpthm}, we obtain
	\begin{align}\notag
		III&\leq 2\mu c_{\ref{lpthm1}}^2\int_{s_{n-1}}^{s_\infty}\Sob{ w_{\lam_n}(\tau)}{L^p}\Sob{\phi_n(w)(\tau)}{L^{p'}}\ d\tau\notag\\
		&\leq2(2\pi)^{2/p}\mu c_{\ref{lpthm1}}^2M\int_{s_{n-1}}^{s_\infty}\left(\int_{\T^2} | \phi_n(w)(\tau)|^{p'}\ {{{}}}\right)^{1/p'} d\tau.\notag
	\end{align}
Then proceeding as in the proof of Lemma \ref{prop:gen:ineq}, we may interpolate with $P=p'$ and $Q=2+p'(\gam-1)$, and apply a Sobolev embedding to arrive at
	\begin{align}\notag
		III\leq C\frac{\mu}{\kap} \frac{2^{n(2/p'+\gam-1)}}{M^{2/p'+\gam-2}}U_{n-1}^{1/p'+\gam/2-1}U_{n-1}.
	\end{align}
We estimate $IV$ similarly.

Therefore, upon combining $I-IV$ and using the fact that $M\geq1$, we arrive at the following nonlinear iteration inequality
	\begin{align}\label{nonlin:iter1}
		U_n&\leq  \frac{C}{\de_\infty\kap}\frac{2^{n(\gam+1)}}{M^{\gam}}U_{n-1}^{1+\gam/2}+ C {2^{n(1+\gam-2/p)}}\left(\frac{F_{L^p}+ \frac{\mu}{\kap}\rho_0}{M^{1+\gam-2/p}}+\frac{\frac{\mu}{\kap}}{M^{\gam-2/p}}\right)U_{n-1}^{1+\gam/2-1/p},\notag\\
			&\leq \frac{CA2^{n(\gamma+1)}}{M^{\gamma-2/p}}(U_{n-1}^{1+\gam/2}+U_{n-1}^{1+\gam/2-1/p}),
	\end{align}
where 
	\begin{align*}
		A:= \max \left \{	\frac{1}{\de_\infty \kap},	F_{L^p}+ \frac{\mu}{\kap}\rho_0 + \frac{\mu}{\kap}	\right \}.
	\end{align*}

We claim that $M$ can be chosen large enough, depending on $U_0$ and $U_{1}$, so that $U_n\goesto0$ as $n\goesto\infty$. Note that if $U_{0}=0$, then we get an $L^{\infty}$ bound automatically. Thus we assume that $U_{0}>0$ in the subsequent discussion. We apply Lemma \ref{lem:nonlin:iter} with 
	\begin{align*}
		a=\gam+1, \quad b=\gam - \frac{2}{p}, \quad k=2, \quad d_{1}=1+\frac{\gamma}{2}, \quad d_{2}=1+\frac{\gam}{2}-\frac{1}{p}.
	\end{align*}
Note that $d=1+\frac{\gam}{2}-\frac{1}{p}>\frac{3}{2}$ due to the Standing Hypotheses $(H3)$. In addition, $U_{n}$ is a non-increasing sequence, thus in particular $U_{1}\leq U_{0}$. By Lemma \ref{lem:nonlin:iter}, we obtain that if
	\begin{align} \label{Mcond}
		M\geq  C\left (A2^{\frac{4(\gam+1)(1+\frac{\gam}{2}-\frac{1}{p})}{\gam-\frac{2}{p}-1}}\left (U_{0}^{\frac{\gamma}{2}}+U_{0}^{\frac{\gam}{2}-\frac{1}{p}}\right )\right )^{\frac{1}{\gam- \frac{2}{p}}},
	\end{align} 
then 
	\begin{align*}
		U_{n}\leq \frac{U_{0}}{2^{ny}},
	\end{align*} 
for all $n\geq 2$, where $y=\frac{3\gam+3}{\gam-\frac{2}{p}}>0$. Thus, $U_{n}\rightarrow \infty$, as $n\rightarrow 0$.

The fact that $U_n\goesto0$ as $n\goesto\infty$ then implies that
	\begin{align}\label{linfty}
		|w(x,k+\de_\infty)|=|w(x,s_\infty)|\leq M,\quad a.e.\ x\in \T^2.\quad
	\end{align}
Since $k\in\R$ was taken fixed and arbitrary, we have that \req{linfty} holds for every $k\in\R$. Hence, $\Sob{w(s)}{L^\infty}\leq M$, for all $s\in\R$, as desired.

	


\subsection{Refinement of $L^p$, $H^\s$ estimates}\label{sect:refine}

In this section, we show how the $L^p$ estimate can be refined to be independent of \textit{both} $m$ and $\mu$, provided that $m$ is chosen large enough.  This is afforded by the $m$-independent $L^\infty$ estimate provided by Proposition \ref{prop:dg}.  Once we have done this, we can similarly refine the estimates provided by Proposition \ref{prop:hs:a priori}.  As before, we will assume that the Standing Hypotheses $(H1)-(H7)$ hold throughout and that $w^\eps$ is the solution of \req{nd:sqg:fbPar}.

Let $G_{L^2}$ be given as in \req{g2}, $F_{H^{\s-\gam/2}}$ by \req{R:sig},  $\Tht_\s$ by \req{sqg:hs:ball}, and $G_{L^p}, M_\infty$ by \req{Gp}.   Define
	\begin{align}\label{gsigtil}
		G_{\s,\infty}^2:= C\left({{F}_{H^{\s-\gam/2}}^2}+\Tht_{\s}^2+\left(\frac{M_\infty}{\kap}\right)^{\frac{2\s}{\gam-1-2/p}}G_{L^2}^2\right).
	\end{align}

\begin{prop}\label{prop:goodlp}
Assume that $(H1)-(H7)$ hold and let $w^\eps$ be the solution of \req{nd:sqg:fbPar}.  There exist absolute constants $C_0, c_0>0$, with $c_0$ depending on $C_0$ such that if \req{gsigtil} holds with $C=C_0$, and  $\mu$ and $m$ satisfy 
	\begin{align}\label{lp:refine:cond}
		2^{m\gam}\geq c_0\frac{\mu}\kap\quad\text{and}\quad 2^m\geq c_0\left(\frac{G_{\s,\infty}}{G_{L^p}}\right)^{1/(1-2/p-\s)},
	\end{align}
then
	\begin{align}\label{good:lp}
		\sup_{s\in\R}\Sob{w^\eps(s)}{L^p}\leq C_0G_{L^p}.
	\end{align}
\end{prop}

\begin{proof}
Let $s_0\in\R$.  For convenience, we denote $w^\eps$ simply by $w$.  We multiply \req{sqg:daPar} by $ w|w|^{p-2}$ and integrate over $\T^2$ to obtain
	\begin{align}\label{lp:balance1}
		\frac{1}{p}\frac{d}{ds}\Sob{ w}{L^p}^p-&\eps\int_{\T^2}\De ww|w|^{p-2}\ {{{}}}+\kap\int_{\T^2} w|w|^{p-2}\Lam^{\gam}w\ {{{}}}\notag\\
				&=\int_{\T^2} f w|w|^{p-2}\ {{{}}}-\mu\int_{\T^2} w|w|^{p-2}S_mw\ {{{}}}+\mu\int_{\T^2} w|w|^{p-2}S_mv\ {{{}}}.
	\end{align}
Observe that an integration by parts yields
	\begin{align}\label{laplace:pos}
		-\int_{\T^2} w|w|^{p-2}\De w\ {{{}}}=(p-1)\int_{\T^2}|\del w|^2|w|^{p-2}\ {{{}}}\geq0.
	\end{align}
We use the positivity of $\Lam^\gam$ from Proposition \ref{lb} and of $-\De$ from \req{laplace:pos}, then extract a damping term from the interpolant operator on the right-hand side, and apply H\"older's inequality to the term with $f$ to obtain
	\begin{align}\label{lp:balance2}
		\frac{1}p\frac{d}{ds}\Sob{ w}{L^p}^p+\mu\Sob{ w}{L^p}^p&\leq\Sob{f}{L^p}\Sob{ w}{L^p}^{p-1}+\mu\int_{\T^2} (T_m w) w^{p-1}\ {{{}}}+\mu\int_{\T^2} (S_m v) w^{p-1}\ {{{}}},\notag\\
		&\leq I+II+III,
	\end{align}
where $T_m$ is defined in \req{lp:block:def}.

We estimate $III$ as
	\begin{align}
		III&\leq\mu\Sob{S_mv}{L^p}\Sob{w^{p-1}}{L^{p/(p-1)}}\leq C\mu \Tht_{H^\s}\Sob{ w}{L^p}^{p-1}\label{inter:tht1}.
	\end{align}
We estimate $II$ with Proposition \ref{prop:hs:a priori} and the Bernstein inequalities as
	\begin{align}
		II&\leq\mu\sum_{j>m}\Sob{\lpj w}{L^p}\Sob{ w}{L^p}^{p-1}\notag\\
					&\leq C\mu\sum_{j>m}2^{j(1-2/p-\s)}(2^{j\s}\Sob{\lpj w}{L^2})\Sob{ w}{L^p}^{p-1}\notag\\
					&\leq C\mu A_{\s,p}(m)\Sob{w}{H^\s}\Sob{ w}{L^p}^{p-1}\label{inter:eta2},
	\end{align}
where
	\begin{align}\label{good:constant}
		A_{\al,r}(m):=
2^{m(1-2/r-\al)}.
	\end{align}
Note that $(H3)$ imposes $1-\s<2/p$, which ensures that $A_{\s,p}(m)\goesto0$ as $m\goesto\infty$.  Also, observe that by Proposition \ref{prop:hs:a priori}, Corollary \ref{coro:dg}, and \req{gsigtil}, we have
	\begin{align}\notag
		\Sob{w(s)}{H^\s}&\leq{\til{G}_{H^\s}}\leq G_{\s,\infty},\quad\forall s\in\R.
	\end{align}
Hence, by applying $I-III$ in \req{lp:balance2} we obtain
	\begin{align}\notag
		\frac{d}{ds}\Sob{ w(s)}{L^p}+\mu\Sob{ w(s)}{L^p}\leq &\Sob{f}{L^p}+C\mu \Tht_{H^\s}+C\mu A_{\s,p}(m)G_{\s,\infty}, \quad\forall s>s_0.
	\end{align}
%

From \req{Gp} and Gronwall's inequality applied over $[s_0, s]$ we have
	\begin{align}\notag
		\Sob{ w(s)}{L^p}\leq & \Sob{w(s_{0})}{L^p}e^{-\mu(s-s_0)}+C\left(G_{L^p}+ A_{\s,p}(m)G_{\s,\infty}\right)(1-e^{-\mu(s-s_{0})}),\quad s\in\R.
	\end{align}
Therefore, by applying Corollary \ref{coro:dg} to bound $\Sob{w(s_0)}{L^p}$ and having chosen $m$ such that the second condition in \req{lp:refine:cond} holds, we arrive at  
	\begin{align}\label{goodlp2}
		\Sob{w(s)}{L^p}\leq CG_{L^p},\quad s>s_0+\mu^{-1}\log\left(\frac{CM_\infty}{G_{L^p}}\right),
	\end{align}
for some absolute constant $C>0$.  Since $s_0\in\R$ is arbitrary, we may send $s_0\goesto-\infty$, which completes the proof.
\end{proof}

Combining Proposition \ref{prop:goodlp} with Proposition \ref{prop:hs:a priori}, we immediately obtain $m$-independent bounds in $H^\s$.  We point out that these bounds are also $\mu$-independent since $\mu$ is taken to be large with respect to $\kap$ (see \req{MainCondition}).

Let $G_{L^2}$ be given as in \req{g2}, $F_{H^{\s-\gam/2}}$ by \req{R:sig},  $\Tht_\s$ by \req{sqg:hs:ball}, and $G_{L^p}$ by \req{Gp}. Define
	\begin{align}\label{ghs}
		G_{H^\s}^2:=C\left({{F}_{H^{\s-\gam/2}}^2}+\Tht_{H^\s}^2+\left(\frac{G_{L^p}}{\kap}\right)^{2\s/(\gam-1-2/p)}G_{L^2}^2\right).
	\end{align}

\begin{coro}\label{coro:goodhs}
Assume the hypotheses of Proposition \ref{prop:goodlp}.  Then there exist absolute constants $C_0, c_0>0$, with $c_0$ depending on $C_0$, such that if \req{ghs} holds with $C=C_0$ and $\mu, m$ satisfy \req{lp:refine:cond}, then
	\begin{align}\label{goodhs:bound}
		\sup_{s\in\R}\Sob{w^\eps(s)}{H^\s}+\int_{-\infty}^\infty\Sob{w^\eps(s)}{H^{\s+\gam/2}}^2\ ds\leq G_{H^\s}.
	\end{align}
\end{coro}

\begin{rmk}
With Corollary \ref{coro:goodhs}, we have furnished bounds for $w^\eps$, which are independent of $\eps$.  It is precisely these bounds, in conjunction with the Aubin-Lions lemma and a diagonal argument, that we invoke in Step 3 of the proof of Proposition \ref{da:exist}, that allow us to deduce the existence of a subsequence $(w^{\eps_j})_{j>0}$ that converges to some $w\in C(\R; V_{\til{\s}})$, for $\til{\s}<\s$, and satisfies $w\in L^\infty(\R;V_\s)\cap L^2(\R;V_{\s+\gam/2})$.
\end{rmk}

\subsection*{Acknowledgments}
The authors would like to thank the Instituto Nacional de Matem\`atica Pura e Aplicada (IMPA) in Rio de Janeiro, Brazil where the Fourth Workshop on Fluids and PDE in 2014 was held and where this work found its conception.  The authors would also like to thank A. Cheskidov for his insightful discussion in the course of this work.  M.S.J. was supported by NSF grant DMS-1418911 and the Leverhulme Trust grant VP1-2015-036.  The work of E.S.T. was supported in part by the ONR grant N00014-15-1-2333 and the NSF grants DMS-1109640 and DMS-1109645.  E.S.T. is also thankful to the warm hospitality of ICERM, Brown University, where this work was completed, during Spring 2017.

\appendix


\section{Proofs of Lemmas \ref{prop:gen:ineq} and \ref{lem:nonlin:iter}}

We now supply the proofs of Lemmas \ref{prop:gen:ineq} and \ref{lem:nonlin:iter}.  For convenience, we also restate them here. 
\begin{lemm}\label{prop:GI}
Suppose $\{g_n\}, \{h_n\}, \{\chi_{g_n}\}, \{\chi_{h_n}\}$ are families of non-negative functions over $\R$ and for all $n\geq0$ satisfy: 
	\begin{enumerate}
		\item$g_n\leq g_{n-1}$ and $h_n\leq h_{n-1}$;
		\item $\chi_{g_{n+1}}\leq C_ng_{n}$ and $\chi_{h_{n+1}}\leq C_nh_{n}$, for some absolute constant $C_n>0$;
		\item $g_{n}(\xi)\neq0$ implies $\chi_{g_n}(\xi)\geq1$;
		\item $h_n(\xi)\neq0$ implies $\chi_{h_n}(\xi)\geq1$.
	\end{enumerate}
Let $1<P, Q<\infty$, $\alpha\in(0,1)$, and $\gam\in(1,2)$ satisfy
	\begin{align}\label{alpha:a}
		\frac{1}{P+Q}=\frac{1}2(1-\alpha)+\frac{(2-\gam)}{4}\alpha.
	\end{align}
Then
	\begin{align}\label{gen:ineq:a}
		\Sob{(g_n,h_n)}{L^P}^P\leq C_n^{Q}\Sob{(g_{n-1},h_{n-1})}{L^2}^{(1-\alpha)(P+Q)}\Sob{(g_{n-1},h_{n-1})}{H^{\gam/2}}^{\alpha(P+Q)}.
	\end{align}
\end{lemm}

\begin{proof}
Observe that
	\begin{align}\notag
		\begin{split}
		\Sob{(g_n,h_n)}{L^P}^P&=\int \left(g_n^{2}+h_n^{2}\right)^{P/2}\\
				&\leq\int \left(g_{n-1}^{2}\chi_{g_n}^{2Q/P}+h_{n-1}^{2}\chi_{h_n}^{2Q/P}\right)^{P/2}\\
				&\leq C_n^Q\int\left(g_{n-1}^{2(1+Q/P)}+h_{n-1}^{2(1+Q/P)}\right)^{P/2}\\
				&\leq C_n^Q\int \left(g_{n-1}^2+h_{n-1}^2\right)^{(P+Q)/2}\\
				&\leq C_n^Q\Sob{(g_{n-1},h_{n-1})}{L^{P+Q}}^{P+Q}.
		\end{split}
	\end{align}
With $\alpha\in(0,1)$ satisfying \req{alpha:a}, by interpolation we have
	\begin{align}\notag
		\Sob{(g_{n-1},h_{n-1})}{L^{P+Q}}\leq\Sob{(g_{n-1},h_{n-1})}{L^2}^{1-\alpha}\Sob{(g_{n-1},h_{n-1})}{L^{4/(2-\gam)}}^{\alpha},
	\end{align}
Thus, by a Sobolev embedding
	\begin{align}\notag
			\begin{split}
			&\Sob{(g_n,h_n)}{L^P}^P\leq C_n^Q\Sob{(g_{n-1},h_{n-1})}{L^2}^{(1-\alpha)(P+Q)}\Sob{(g_{n-1},h_{n-1})}{H^{\gam/2}}^{\alpha(P+Q)},
			\end{split}
	\end{align}
which is precisely \req{gen:ineq:a}.
\end{proof}

\begin{lemm}
Let $\{V_n\}_{n\geq0}$ be a sequence of positive numbers.  Suppose there exist $M>0, a>0, b>0, K>0$, and $d_{j}>0$ , for $j=1,2,\dots,K$, such that $d=\min\{d_{1}, d_{2}, ..., d_{K}\}>\frac{3}{2}$ and
	\begin{align*}
		V_{n}\leq C\frac{2^{na}}{M^{b}}\sum_{j=1}^{K} V_{n-1}^{d_{j}},
	\end{align*}
holds for all $n\geq 1$.  Let 
	\begin{align*}
		y_0:= \frac{3a}{2d-3}.
	\end{align*} 
There exists an absolute constant $c_0>0$ such that if $M$ satisfies
	\begin{align*}
		M\geq  C_0\max \left \{\left (\frac{2^{2(a+y_0)}\sum_{j=1}^{K}V_{1}^{d_{j}}}{V_{0}}\right )^{\frac{1}{b}},\left (\frac{\sum_{j=1}^{K}V_{0}^{d_{j}}}{V_{0}}\right )^{\frac{1}{b}} \right \},
	\end{align*} 
then
	\begin{align*}
		V_{n}\leq \frac{V_{0}}{2^{ny_0}}, 
	\end{align*}
holds for all $n\geq 2$.  Moreover, if $V_{1}\leq V_{0}$ then it suffices to choose $M$ to satisfy
	\begin{align*}
		M\geq C_0\left (2^{2(a+y_0)}\sum_{j=1}^{K}V_{0}^{d_{j}-1}\right )^{\frac{1}{b}}.
	\end{align*} 			
\end{lemm}

\begin{proof}
We prove the claim by induction. First, we show that the assertion is correct at $n=2$.  
	\begin{align*}
		V_{2}&\leq \frac{C2^{2a}}{M^{b}}\sum_{j=1}^{K} V_{1}^{d_{j}}\\
			&\leq \frac{V_{0}}{2^{2y_0}}\frac{C2^{2(a+y)}\sum_{j=1}^{K} V_{1}^{d_{j}}}{V_{0}}\frac{1}{M^{b}}\\
			&\leq \frac{V_{0}}{2^{2y_0}},
	\end{align*}
where the last inequality is due to the choice of $M$. We assume the assertion is correct at step $n$. Then, 
	\begin{align*}
		V_{n+1}&\leq \frac{C2^{(n+1)a}}{M^{b}}\sum_{j=1}^{K} V_{n}^{d_{j}}\\
			    &\leq \frac{C2^{(n+1)a}}{M^{b}}\sum_{j=1}^K (\frac{V_{0}}{2^{ny_0}})^{d_{j}}\\
			    &= \frac{C2^{(n+1)a}}{M^{b}2^{ny_0d}}\sum_{j=1}^{K} V_{0}^{d_{j}}\\
			    &\leq \frac{V_{0}}{2^{(n+1)y_0}}\frac{C\sum_{j=1}^{K} V_{0}^{d_{j}}}{V_{0}}\frac{1}{M^{b}}\\
			    &\leq  \frac{V_{0}}{2^{(n+1)y_0}},
	\end{align*}
where we use the definition of $y$ and the fact that $d>\frac{3}{2}$. The last inequality is due to the choice of $M$.
If we also assume that $V_{1}\leq V_{0}$, then 
	\begin{align*}
		M\geq C\left (2^{2(a+y_0)}\sum_{j=1}^{K}V_{0}^{d_{j}-1}\right )^{\frac{1}{b}},
	\end{align*} 
implies that 
	\begin{align*}
		M\geq  \max \left \{C\left (\frac{2^{2(a+y_0)}\sum_{j=1}^{K}V_{1}^{d_{j}}}{V_{0}}\right )^{\frac{1}{b}},  C\left (\frac{\sum_{j=1}^{K}V_{0}^{d_{j}}}{V_{0}}\right )^{\frac{1}{b}} \right \}.
	\end{align*} 
\end{proof}

\section{Existence of time-periodic solutions, steady state solutions, and determining modes \& volume elements}\label{app}

Here we provide elementary proofs of the existence of time-periodic solutions given a time-periodic force, the existence of steady state solutions when $f$ is time-independent, and the existence of finitely many {determining parameters} for \req{sqg}.  These results complement those found in \cite{dai:steady} and \cite{chesk:dai:subcritqg}, where existence of steady state solutions is established in the case where the domain is given by the whole space, $\R^2$, and the existence of finitely many determining modes is shown using an approach based on the Kolmogorov dissipation wavenumber, respectively.

\subsection{Time-periodic and steady state solutions}
We will prove the following theorem.

\begin{thm}\label{thm:timeper}
Let $\gam\in(1,2)$, $\s>2-\gam$, and $p\in(1,\infty]$ such that $1-\s<2/p<\gam-1$.  Let $f\in L^\infty(0,\infty;V_{\s-\gam/2}\cap L^p(\T^2))$, where $\mathcal{Z}$ is as in \req{Z:space}.  Suppose that there exists $\tau_f>0$ such that $f(t)=f(t+\tau_f)$ a.e. $t\in\R$.  Then there exists $\tht\in L^\infty(0,\infty;V_\s)\cap L^2(0,\infty;V_{\s+\gam/2})$ satisfying \req{sqg} such that $\tht(\cdotp,t)=\tht(\cdotp, t+\tau_f)$ a.e. $t\geq0$.
\end{thm}

To prove this theorem, let us recall the following well-posedness result from \cite{const:wu:qgwksol}.

\begin{prop}[Global existence]\label{prop:glob}
Let $1<\gam\leq 2$, and $\s>2-\gam$.  Given $T>0$, suppose that 
$\tht_0\in V_\s$ and $f$ satisfy
	\[
		f\in L^2(0,T; V_{\s-\gam/2})\cap
                L^1(0,T;{L}_{per}^p(\T^2)),
	\]
where $1-\s\leq2/p<\gam-1$.  Then there is a weak solution $\tht$ of
\req{sqg} such that
	\begin{align}\notag
		\tht\in L^\infty(0,T;V_\s)\cap
                L^2(0,T;V_{\s+\gam/2}).
	\end{align}
\end{prop}

\begin{prop}[Uniqueness]\label{prop:uniq}
Let $T>0$ and $1<\gam\leq2$.  Suppose that $\tht_0\in
{L}_{per}^2(\T^2)\cap\mathcal{Z}$ and $f\in L^2(0,T;V_{-\gam/2})$.  Then for $p\geq1, q>0$ satisfying
	\[
		\frac{1}p+\frac{\gam}{2q}=\frac{\gam-1}2,
	\]
there is at most one solution to \req{sqg} such that
	\[
		\tht\in L^q(0,T;{L}_{per}^p(\T^2)).
	\]
\end{prop}

Now let us prove Theorem \ref{thm:timeper}.  We remark that our strategy is only one way, albeit a cheap one, to prove the existence of a time-periodic solution.

\begin{proof}[Proof of Theorem \ref{thm:timeper}]
Firstly, Propositions \ref{prop:glob} and \ref{prop:uniq} together imply that the corresponding solution operator, $S(t;t_0)\tht_0=\tht(t;t_0)$, which denotes the solution at time $t>t_0\geq0$ to \req{sqg} with $S(t_0;t_0)\tht_0=\tht_0$ and source term initialized at $f|_{t=t_0}$, is well-defined.  In fact, upon inspection of the proof of Proposition \ref{prop:ga} in \cite{ju:qgattract}, we also have $S:H^\s\goesto H^{\s}$ is continuous for fixed $t,t_0\geq0$.  Since $f\in L^\infty(0,\infty;V_{\s-\gam/2}\cap L^p)$, the same analysis as the proofs of Propositions \ref{prop:sqg:ball} and \ref{prop:ga} can be applied to establish the bounds \req{fp}, \req{fgam}.  In particular, we have that for each $t\geq0$, $S(t;0):\mathcal{B}_{H^\s}\goesto\mathcal{B}_{H^\s}$.  Since $S(t;0)\tht_0\in L^2_{loc}(0,\infty;H^{\s+\gam/2})$, for each $\tht_0\in H^\s$, it follows that $S(t;0)\tht_0\in H^{\s+\gam/2}$ for a.e. $t>0$, where $H^{\s+\gam/2}$ is compactly imbedded in $H^\s$ by the Rellich compactness lemma.  On the other hand, if we consider $S(t+s,s)\tht_0$ for $s,t>0$, then we have $S(t+s,s)\tht_0\in H^{\s+\gam/2}$, which is again compactly imbedded in $H^\s$.

Hence, by the Schauder fixed point theorem, there exists $\tht_*\in\mathcal{B}_{H^\s}$ such that
	\begin{align}\notag
		\tht_*=S(\tau_f;0)\tht_*.
	\end{align}
Let $\tht\in L^\infty_{loc}(0,\infty;V_\s)\cap L^2_{loc}(0,\infty;V_{\s+\gam/2})$ denote the unique solution of \req{sqg} corresponding to initial data $\tht(\cdotp, 0)=\tht_*(\cdotp)$.  Since $f(t)=f(t+\tau_f)$, for all $t\geq0$, it then follows from uniqueness (Proposition \ref{prop:uniq}) that
	\begin{align}
		S(t;0)\tht_*=S(t+\tau_f;\tau_f)\tht_*\quad\text{and}\quad S(t+\tau_f;\tau_f)S(\tau_f,0)\tht_*=S(t+\tau_f;0)\tht_*.
	\end{align}
Therefore, invoking the fact that $\tht_*=S(\tau_f;0)\tht_*$, we deduce that
	\begin{align}
		\tht(t;0)=S(t;0)\tht_*=S(t+\tau_f;\tau_f)\tht_*=S(t+\tau_f;\tau_f)S(\tau_f;0)\tht_*=S(t+\tau_f;0)\tht_*=\tht(t+\tau_f;0)
	\end{align}
for all $t\geq0$, as claimed.
\end{proof}

Next, we will use Theorem \ref{thm:timeper} to establish existence of steady-state solutions to \req{sqg} in the case that the external source term, $f$, is time-independent.  We stress that this is not the most straightforward way to show existence of steady states, but we use the previous argument about the time periodic case to establish it.

\begin{thm}\label{thm:sss}
Let $\gam\in(1,2)$, $\s>2-\gam$, and $p\in(1,\infty]$ such that $1-\s<2/p<\gam-1$.  Suppose that $f\in V_{\s-\gam/2}\cap L^p$ is time-independent. Then there exists $\Tht\in V_\s$ is a steady state solution of \req{sqg}.
\end{thm}

\begin{proof}
Observe that for each $n>0$, $f=f(t)=f(t+2^{-n})$, for all $t\geq0$.  In particular $f$ can be viewed as a time-periodic function $f\in L^\infty(0,\infty; V_{\s-\gam/2}\cap L^p)$ of period $2^{-n}$.  Then as in the proof of Theorem \ref{thm:timeper}, for each $n>0$, there exists a $\tht_0^{(n)}\in\mathcal{B}_{H^\s}$ satisfying $\tht^{(n)}_0=S(2^{-n};0)\tht_0^{(n)}$ such that the corresponding solution, $\tht_*^{(n)}(t)=S(t;0)\tht_0^{(n)}$, of \req{sqg} satisfies $\tht_*^{(n)}(t)=\tht_*^{(n)}(t+2^{-n})\in\mathcal{B}_{H^\s}$, for all $t\geq0$.  By Rellich's theorem, upon possibly passing to a subsequence, we may assume that $\tht_0^{(n_k)}\goesto \Tht\in \mathcal{B}_{H^\s}$, as $k\goesto\infty$.  Let $\tht_*\in L^\infty_{loc}(0,\infty;V_\s)\cap L^2_{loc}(0,\infty;V_{\s+\gam/2})$ denote the unique strong solution of \req{sqg} corresponding to initial data given by $\tht_*(\cdotp,0)=\Tht(\cdotp)$.

For $m>0$, we define $t_m:=2^{-m}$.  Fix $k>0$ and consider $\ell>k$, so that $n_\ell>n_k$.  Then by Proposition \ref{prop:ga}, we may apply the continuity of the solution operator to obtain
	\begin{align}\notag
		\tht_*(t_{n_k})=S(t_{n_k};0)\Tht=\lim_{\ell\goesto\infty}S(t_{n_k};0)\tht_0^{(n_\ell)}=\lim_{\ell\goesto\infty}S(2^{n_{\ell}-n_k}t_{n_\ell};0)\tht_0^{(n_\ell)}.
	\end{align}
Observe that for $n_\ell-n_k>0$, we have $S(2^{n_{\ell}-n_k}t_{n_\ell};0)\tht_0^{(n_\ell)}=\tht_0^{(n_\ell)}$.  This implies that
	\begin{align}\label{timeper}
		\tht_*(t_{n_k})=\lim_{\ell\goesto\infty}\tht_0^{(n_\ell)}=\Tht,
	\end{align}
which holds for all $k>0$.  In particular, by uniqueness, it follows that
	\begin{align}\label{per1}
		S(t+t_{n_k};0)\Tht=S(t;0)\Tht, \quad t>0,\quad k>0.
	\end{align}
Observe that
	\begin{align}
		\frac{S(t+t_{n_k};0)\Tht-S(t;0)\Tht}{t_{n_k}}=0,\quad t>0,\quad k>0.
	\end{align}
Thus, in the limit as $k\goesto\infty$, we deduce that $\bdy_tS(t;0)\Tht=0$.  In particular, $S(t;0)\Tht=\Tht$, for all $t>0$.  so that $\Tht$ is a solution to \req{sqg} for $t>0$.  Since $\Tht$ is time-independent, we are done.
\end{proof}

\subsection{Finite determining parameters}

First, we define what we mean by a ``determining operator" (see for instance \cite{jones:titi}).

\begin{defn}\label{defn:dmodes}
Let $J_{h}:L^2_{per}(\T^2)\goesto L^2_{per}(\T^2)$ denote a linear operator.  Let $u$ and $v$ be two global solutions of a given system of evolution equations. Then the operator $J_{h}$ is said to be determining if 
	\begin{align}\notag
		\lim_{t\goesto \infty} \Sob{u(t)-v(t)}{L^2}=0,
	\end{align}
whenever 
	\begin{align}\notag
		\lim_{t\goesto \infty} \Sob{J_{h}u(t)-J_{h}v(t)}{L^2}=0.
	\end{align}
\end{defn}

We will additionally require that the operators, $J_h$, satisfy $J_h:\dot{H}^\be(\T^2)\goesto L^2_{per}(\T^2)$ such that
	\begin{align}\label{Ph:dim}
		d(h):=\dim(J_hL^2_{per})<\infty,
	\end{align}
and
\begin{align}\label{TI}
		\Sob{\phi-J_h\phi}{L^2}\leq Ch^{\be}\Sob{\phi}{\dot{H}^\be}\quad\text{and}\quad\Sob{\phi-J_h\phi}{\dot{H}^{-\be}}\leq Ch^{\be}\Sob{\phi}{L^2},\quad \be\in(0,1).
	\end{align}
Important examples of $J_h$ satisfying the conditions \req{Ph:dim} and \req{TI} include the spectral projection onto modes $|\mathbf{k}|\leq 1/h$ or projection onto local spatial averages with linear mesh size $h$, i.e., volume elements projection.  That such examples verify \req{Ph:dim} and \req{TI} has been demonstrated in \cite{jmt:sqgda}. We now state another version of determining projections, which in the case of spectral projection or volume elements projection, is equivalent to Definition \ref{defn:dmodes}.  It is the one we will make use of below in Theorem \ref{prop:detparam} (cf. \cite{fkjt1, fkjt2}).

\begin{defn}\label{defn:dmodes2}
Let $J_{h}$ be as in Definition \ref{defn:dmodes}.  Let $u$ and $v$ be two trajectories on the global attractor of a given system of evolution equations. Then the operator $J_{h}$ is said to be determining if 
	\begin{align}\notag
		u(t)=v(t),\quad t\in\R,
	\end{align}
whenever 
	\begin{align}\notag
		J_{h}u(t)=J_{h}v(t),\quad t\in\R.
	\end{align}
In this case, any basis of $J_hL^2_{per}$ is referred to as a set of determining parameters for the given system.
\end{defn}

\begin{thm}\label{prop:detparam}
Let $1<\gam<2$ and $\s>2-\gam$.  Let $1<p<\infty$ satisfy $1-\s<2/p<\gam-1$.  Let $f\in \dot{H}^{\s-\gam/2}_{per}(\T^2)\cap L^p(\T^2)$ and $\A$ denote the corresponding global attractor of \req{sqg}.  Let
	\begin{align}
		\Tht_{L^p}:=\sup_{\tht\in\A}\sup_{t\in\R}\Sob{\tht(t)}{L^p}.
	\end{align}
Suppose that $J_h$ satisfies \req{TI} for $\be=1-\gam/2$.  Suppose $\tht_1(\cdotp), \tht_2(\cdotp)\sub\A$.  There exists an absolute constant $c_0>0$ such that if $h$ satisfies
	\begin{align}\label{det:modes}
		h^{-1}>c_0\left(\frac{\Tht_{L^p}}{\kap}\right)^{1/(\gam-1-2/p)},
	\end{align}
then $\tht_1(t)=\tht_{2}(t)$, for all $t\in\R$, whenever $J_h\tht_1(t)=J_h\tht_{2}(t)$, for all $t\in\R$. 
\end{thm}

\begin{proof}
Let $J_h$ be given as in Definition \ref{defn:dmodes}.  Suppose that $J_h\tht_1(t)=J_h\tht_2(t)$ for all $t\in\R$.  Without loss of generality, we may assume that $\s<1$.  Otherwise, $\s>1$ and $H^\s\imb H^{\s'}$, for $\s'\leq1$, so that $\tht_1,\tht_2\in L^\infty(\R;H^{\s'})$.  The proof closely follows that of Proposition \ref{Uniqueness} above, except that the nonlinear term proceeds in a slightly different manner.  Indeed, let $\de=\tht_1-\tht_{2}$ and $\psi=-\Lam^{-1}\de$. 
Then $\de$ satisfies
	\begin{align}\notag
		\bdy_t\de+\kap\Lam^\gam\de+\Ri^\perp\de\cdotp\del\de+\Ri^\perp\de\cdotp\del\tht_{2}+u_2\cdotp\del\de=0,
	\end{align}
so that upon taking the $L^2$ scalar product with $\psi$, we obtain
	\begin{align}\label{psi:bal}
		\frac{1}{2}\frac{d}{dt}\Sob{\psi}{H^{1/2}}^2+\kap\Sob{\psi}{H^{\frac{\gam+1}2}}^2=\int (u_2\cdotp\del\de)\psi\ dx.
	\end{align}
We estimate the right-hand side as in $I$ of Proposition \ref{Uniqueness}.  Indeed, we have
	\begin{align}\label{nonlin:det}
		\left|\int (u_2\cdotp\del\de)\psi\ dx\right|\leq C\Tht_{L^p} \|\psi\|^{2}_{H^{1+\frac{1}{p}}}.
	\end{align}
We then interpolate with Proposition \ref{interpol} to obtain
	\begin{align}
		\Sob{\psi}{H^{1+1/p}}\leq\Sob{\psi}{H^{\frac{\gam+1}2}}^{2+2/p-\gam}\Sob{\psi}{H^{\gam/2}}^{\gam-1-2/p}
	\end{align}
Observe that $J_h\de(t)=0$, for all $t\in\R$.  It follows from the identity $\psi=-\Lam^{-1}\de$ and \req{TI} that
	\begin{align}
		\Sob{\psi}{H^{\gam/2}}&=\Sob{\de}{H^{\frac{\gam-2}2}}=\Sob{\de-J_h\de}{H^{-\frac{2-\gam}2}}\leq Ch^{\frac{2-\gam}2}\Sob{\de}{L^2}\notag\\
			&= Ch^{\frac{2-\gam}2}\Sob{\de-J_h\de}{L^2}\leq Ch\Sob{\de}{H^{\gam/2}}=Ch\Sob{\psi}{H^{\frac{\gam+1}2}}.
	\end{align}
Thus, upon returning to \req{nonlin:det} and applying Young's inequality, we obtain
	\begin{align}
		\left|\int (u_2\cdotp\del\de)\psi\ dx\right|\leq C\Tht_{L^p}h^{\gam-1-2/p}\Sob{\psi}{H^\frac{\gam+1}2}^2,
	\end{align}
so that by \req{det:modes} and Poincar\'e inequality we have
	\begin{align}
		\frac{d}{dt}\Sob{\psi}{H^{1/2}}^2+{\kap_0}\Sob{\psi}{H^{1/2}}^2\leq0,
	\end{align}
for some $\kap_0>0$.  By Gronwall's inequality and the Poincar\'e inequality, it follows that
	\begin{align}
		\Sob{\psi(t)}{H^{1/2}}^2\leq\Sob{\psi(s)}{H^{1/2}}e^{-\kap_0(t-s)}\leq C\Tht_{L^2}e^{-\kap_0(t-s)},\quad s\leq t.
	\end{align}
Finally, sending $s\goesto-\infty$ implies that $\psi(t)=0$ for all $t\in\R$, as desired.
\end{proof}

In particular, in the special case where $J_h$ denotes spectral projection, we have that \req{Ph:dim} and \req{TI} hold, so that it immediately follows from Theorem \ref{prop:detparam} that the equation \req{sqg} has finitely many determining modes.

\begin{coro}\label{coro:dm}
Given $h>0$, let $P_h$ denote projection onto Fourier modes up to wavenumbers $|\mathbf{k}|\leq 1/h$.  Assume the hypotheses in Theorem \ref{prop:detparam}.  Then there exists an $h>0$ such that whenever $P_h\tht_1(t)=P_h\tht_2(t)$, for all $t\in\R$, where $\tht_1(\cdotp), \tht_2(\cdotp)\sub\A$, we have $\tht_1(t)=\tht_2(t)$, for all $t\in\R$.
\end{coro}

\begin{rmk}\label{app:rmk}
We point out that the estimate for the number of determining modes implied by \req{det:modes} is essentially optimal in the sense that it matches the scaling of the estimate obtained in \cite{chesk:dai:subcritqg} when one sets $p=\infty$ in \req{det:modes}.  Note that $p=\infty$ is valid when $\s>1$ in light of the Standing Hypotheses.  However, we point out that the convergence of the high modes is obtained in a space of higher regularity in \cite{chesk:dai:subcritqg}, whereas here, the convergence is obtained in a rather weak topology, i.e., $H^{-1/2}$ for $\tht$.  From a practical perspective, the desire for convergence in stronger topologies is clear, but when working on the global attractor this point is irrelevant.  Indeed, we have shown that if sufficiently many low modes of two trajectories of \req{sqg} on the global attractor agree for all time, then they must be identical.
\end{rmk}

\bibliographystyle{plain}

\end{document}